\documentclass[reqno, 12pt]{amsart}

\usepackage{amsmath,graphicx}
\usepackage{adjustbox}
\usepackage{amssymb}
\usepackage{tikz, tikz-cd}
\usepackage{fullpage}
\usepackage[all]{xy}
\usepackage[margin=1in]{geometry}
\usepackage{amsthm}
\usepackage{amsfonts}
\usepackage{bbm}
\usepackage{comment}
\usepackage{amsmath}
\usepackage{amsthm}
\usepackage{bbm}
\usepackage{array} 
\usepackage{color}
\usepackage[utf8]{inputenc}
\usepackage[english]{babel}
\usepackage{hyperref}
\usepackage{amsmath, amssymb, graphics, setspace}
\usepackage{mathrsfs}
\newcolumntype{L}{>{$}l<{$}}

\DeclareMathSymbol{\shortminus}{\mathbin}{AMSa}{"39}

\newtheorem{theorem}{Theorem}[section]
\newtheorem{lemma}[theorem]{Lemma}

\newtheorem{corollary}[theorem]{Corollary}

\newtheorem{proposition}[theorem]{Proposition}
\theoremstyle{definition}  
\newtheorem{definition} [theorem] {Definition} 

\newtheorem{example} [theorem] {Example}
\newtheorem{remark} [theorem] {Remark}

\theoremstyle{definition}

\newcommand{\Z}{{\mathbb{Z}}}

\newtheorem{innercustomthm}{Theorem}
\newenvironment{customthm}[1]
{\renewcommand\theinnercustomthm{#1}\innercustomthm}
{\endinnercustomthm}
\newtheorem{innercustomcoro}{Corollary}
\newenvironment{customcoro}[1]
{\renewcommand\theinnercustomcoro{#1}\innercustomcoro}
{\endinnercustomcoro}

\title{On Acylindrical Tree Actions and  Outer Automorphism Group of Baumslag-Solitar Groups}

\author[Som]{Bratati Som}
\author[Wang]{Daxun Wang}

\address[]{B. Som: Department of Mathematics, University at Buffalo}
\email{bratatis@buffalo.edu}
\address[]{D. Wang: Yau Mathematical Sciences Center, Tsinghua University, Beijing, China}
\email{wangdaxun@mail.tsinghua.edu.cn}

\begin{document}

\begin{abstract}
This paper explores acylindrical actions on trees, building on previous works related to the mapping class group and projection complexes. We demonstrate that the quotient action of a $1$-acylindrical action of a group on a tree by an equivariant family of subgroups remains $1$-acylindrical. We establish criteria for ensuring that this quotient action is non-elementary acylindrical, thus preserving the group’s acylindrical hyperbolicity. Additionally, we show that the fundamental group of a graph of groups admits the largest acylindrical action on its Bass-Serre tree under certain conditions. As an application, we analyze the outer automorphism group of non-solvable Baumslag-Solitar groups, we prove its acylindrical hyperbolicity, highlighting the differences between various tree actions and identifying the largest acylindrical action.
\end{abstract}

\maketitle

\begin{section}{Introduction}

The study of \textit{acylindrical action} on a hyperbolic space was initiated by Bowditch in \cite{Bowditchcurve}. Roughly speaking, it is the action whereby the number of elements that simultaneously coarsely stabilize sufficiently distant points is uniformly bounded. The study of acylindrical action on hyperbolic spaces was developed by Osin \cite{acylhypgps} building on the work of Bowditch \cite{Bowditchcurve} and Sela \cite{Selaacyl}. This has proven to be an effective tool to study many important families of groups. In fact, many groups (such as mapping class groups and the outer automorphism group of free groups, as well as most 3-manifold groups, etc.) admit a non-elementary acylindrical action on some hyperbolic spaces. 

In this paper, we are interested in the quotient action of an acylindrical hyperbolic action. Our motivation for studying this quotient action comes from the algebraic counterpart to Thurston's Dehn filling theorem, in relation to relatively hyperbolic groups \cite{Osinrelativehyp}, \cite{Dehnfillinginrel} and, more generally, to acylindrically hyperbolic groups \cite{DGO}.

Specifically, similar to how new manifolds with hyperbolic structures are created from a 3-manifold with an existing hyperbolic structure via Dehn filling, in the context of group theory, taking quotients of an acylindrically hyperbolic group by certain ``special" normal subgroups produces new acylindrically hyperbolic groups. In \cite{DehnFillingDehnTwist}, Dahmani, Hagen, and Sisto showed that quotients of the mapping class group by normal subgroups generated by suitable powers of Dehn twists preserve the acylindricity of the group. In proving this theorem, the authors relied on the fact established by Dahmani \cite{Dahmani} that the curve graph has the structure of a composite projection graph. Later, Clay and Mangahas in \cite{hyperbolicquotient} investigated a scenario where groups act on projection complexes with a sufficient number of WPD hyperbolic elements. They demonstrated the acylindricity of the quotient group by normal closure of subgroup generated by appropriate \textit{``equivariant L-spinning family."} As an application of this construction, Clay and Mangahas present new examples of acylindrically hyperbolic groups obtained from quotients of mapping class groups of orientable surfaces. More generally, for actions on hyperbolic spaces, Dahmani, Guirardel and Osin \cite{DGO} introduced the notion of \textit{very rotating families} which provides a natural framework for developing a geometric version of small cancellation theory. They prove similar quotient results using normal subgroups generated by the rotating families. Even though simplicial trees can be seen as projection complexes naturally (see Section 3.1, \cite{CMM}) as well as $0$-hyperbolic spaces, none of the above mentioned theories apply to trees directly. In this paper, we not only develope a similar theory for actions on trees, but our conditions are more general as we drop both rotating and spinning condition and work with only an equivariant family of subgroups. We recall the definition of an equivariant family of subgroups below:

Let $G$ be a group acting on a tree $T$. For each vertex $v$ of $T$, let $R_v$ be a subgroup of the stabilizer of $v$ in $G$. We say the family of subgroups $\{R_v\}$ is an \textit{equivariant family} of subgroups of $G$ if $gR_vg^{-1}=R_{gv}$, for every $g$ in $G$ and every vertex $v$ of $T$.

We note that, by Bass-Serre theory, any group $G$ acting on a tree without edge inversions induces a \textit{graph of groups} $(\Gamma,\mathcal{G})$ of $G$, which consists of a graph $\Gamma$ together with a group for each vertex and edge of $\Gamma$, and
monomorphisms from each edge group to the adjacent vertex groups. Conversely, any graph of groups $(\Gamma,\mathcal{G})$ has a canonical associated group $\pi_1(\Gamma,\mathcal{G})$, called the \textit{fundamental group of graph of groups}, and a tree $X$, called the \textit{Bass-Serre tree}, such that $\pi_1(\Gamma,\mathcal{G})$ acts on the Bass-Serre tree $X$. The fundamental theorem of Bass-Serre theory says that these processes are in fact mutually inverse. 

We say an action of a group on a tree is an \textit{acylindrical tree action} if it is acylindrical. In \cite{acylactionontrees}, Minasyan and Osin showed that a tree action is acylindrical if and only if it is \textit{$(k,C)$-acylindrical} in the sense that the
pointwise stabilizer of any edge path in the tree of length at least $k$ contains at most $C$ elements. In this paper, we first show that the quotient action of a $(1,C)$-acylindrical tree action by an equivariant family of subgroups is always a $(1,C)$-acylindrical tree action (see Theorem \ref{thm: quotientActionAcyl}). Further, we provide sufficient conditions for an equivariant family of subgroups that ensure the quotient action is \textit{non-elementary acylindrical} and as a consequence the quotient group remains acylindrically hyperbolic.

\begin{customthm}{A}\label{thmA}
\textit{Let $G$ be a group acting on a simplicial tree $T$, and let $N$ be a normal subgroup of $G$. If $N$ is a normal subgroup of $G$ generated by an equivariant family of subgroups $\{R_v\}$, then the quotient space $T/N$ is a tree. Moreover, suppose the action $G\curvearrowright T$ is a $(1,C)$-acylindrical action. Then the quotient action $G/N\curvearrowright T/N$ is a non-elementary acylindrical tree action if $\{R_v\}$ satisfies at least one of the following conditions:}
\begin{enumerate}
    \item \textit{the first Betti number $b_1(\Gamma)\geq 2$ where $\Gamma=T/G$.}
    \item \textit{there are two vertices $\hat{u}$ and $\hat{v}$ in $\Gamma$ such that $\langle G_{\hat{e}},R_{\hat{u}}\rangle\neq G_{\hat{u}}$ and $\langle G_{\hat{f}},R_{\hat{v}}\rangle\neq G_{\hat{v}}$, where $\hat{e}$ and $\hat{f}$ are the edges on the geodesic path $[\hat{u},\hat{v}]$ adjacent to $\hat{u}$ and $\hat{v}$ respectively. Further, there exist two elements $g,h\in G_{\hat{v}}\setminus \langle G_{\hat{f}}, R_{\hat{v}}\rangle$ such that $gh^{-1}\not\in \langle G_{\hat{f}}, R_{\hat{v}}\rangle$.}
    \item \textit{$\Gamma$ contains a immersed circuit. Further, there exists a vertex $\hat{v}$ and an edge $\hat{e}$ adjacent to $\hat{v}$ on the circuit such that $\langle G_{\hat{e}}, R_{\hat{v}}\rangle\neq G_{\hat{v}}$.}
\end{enumerate}
\end{customthm}

Two types of graphs of groups are the simplest interesting examples. One is called an \textit{edge of groups}, where the underlying graph consists of a single edge. The other one is called a \textit{loop of groups}, where the underlying graph consists of a single loop. By Bass-Serre theory, the fundamental group of an edge of groups is a free amalgamated product, while the fundamental group of a loop of groups is a HNN-extension. As a consequence of Theorem \ref{thmA}, we obtain the following two corollaries about quotient actions for these two special cases.

\begin{customcoro}{B}\label{coroB}
\textit{Let $G=A\ast_C B$ be a free amalgamated product where $C$ is finite, and let $T$ be its Bass-Serre tree. If $N$ is generated by an equivariant family of subgroups $\{R_v\}$. Then the quotient action $G/N\curvearrowright T/N$ is a non-elementary acylindrical tree action if $\{R_v\}$ satisfies the following:}
\begin{enumerate}
    \item \textit{$\langle C, R_A\rangle\neq A$ and  $\langle C, R_B\rangle\neq B$; and}
    \item \textit{there exists two different elements $g$ and $h$ in $ B\setminus \langle C, R_B\rangle$ such that $gh^{-1}\not\in \langle C, R_B\rangle$.}
\end{enumerate}
\end{customcoro}

\begin{customcoro}{C}\label{coroC}
\textit{Let $G=A\ast_{H\sim_{\phi} K}$ be an HNN-extension where $H,K$ are finite, and let $T$ be its Bass-Serre tree. If $N$ is generated by an equivariant family of subgroups $\{R_v\}$. Then the quotient action $G/N\curvearrowright T/N$ is a non-elementary acylindrical tree action if $\{R_v\}$ satisfies either $\langle H, R_A \rangle\neq A$ or $\langle K, R_A \rangle\neq A$.}
\end{customcoro}

Note that any given group with an acylindrical action may admit many acylindrical actions on different hyperbolic spaces. It is natural to ask which acylindrical action is the ``best." In \cite{hyperbolicstructure}, Abbott, Balasubramanya, and Osin introduced a partial order on the set of acylindrical actions which allows us to optimize all acylindrical actions. In particular, a larger action in this poset reveals, to some extent, more information about the group. Therefore, it is natural to consider whether a largest action exists in this poset. 
It is proven that every hierarchically hyperbolic group \cite{largestaction} and many graph products \cite{graphproduct} admit a largest acylindrical hyperbolic action. In the second part of this paper, we study existence of such largest acylindrical actions that arise from graphs of groups. In \cite{acylactionontrees}, Minasyan and Osin provided sufficient conditions for the tree action $\pi_1(\Gamma,\mathcal{G})\curvearrowright X$ to be acylindrical. In this paper, we give sufficient conditions for the fundamental group of graph of groups $\pi_1(\Gamma, \mathcal{G})$ to admit a largest acylindrical action.

\begin{customthm}{D}\label{thmD}
\textit{Let $(\Gamma, \mathcal{G})$ be a finite graph of groups, and let $T$ be the Bass-Serre tree of $(\Gamma, \mathcal{G})$. Denote the fundamental group $\pi_1(\Gamma, \mathcal{G})$ by $G$. If the action $G\curvearrowright T$ is acylindrical, and if each vertex group acts elliptically whenever $G$ acts acylindrically on any hyperbolic space $S$, then $G$ admits the largest acylindrical action on $T$.}   
\end{customthm}

We remark that the above theorem shares a similar idea as Proposition 4.13 of \cite{hyperbolicstructure}. However, when restricting to tree actions which are specifically acylindrical, we obtained a proof of existence of a largest such action by using Bass-Serre theory.

Next, as an application of the previously mentioned results along with some additional findings, we concentrate on a specific group that exhibits interesting graph of groups structures: the outer automorphism group of non-solvable Baumslag-Solitar groups. These groups have been studied particularly in \cite{Autpdividesq}, \cite{AutomorphismBSgroups}, \cite{pqnocommonfactors} and \cite{pdoesnotdivideq}. Recall that Baumslag-Solitar groups have the following standard presentation
\[\mbox{BS}(p,q) = \langle x, t \hspace{0.2cm}| \hspace{0.2cm} tx^pt^{-1} = x^q \rangle\]
where $p,q$ are both non-zero integers. For non-solvable Baumslag-Solitar groups neither $|p|$ nor $|q|$ equals 1. In general, acylindricity of the automorphism group Aut$(G)$ has been studied by others with some conditions on $G$. Genevois \cite{oneendacyl} proved that if $G$ is a one-ended hyperbolic group then Aut$(G)$ is acylindrically hyperbolic. Genevois and Horbez \cite{infiniteendacyl} showed that if $G$ is an infinitely-ended finitely generated group then Aut$(G)$ is acylindrically hyperbolic. For the outer automorphism group Out$(G)$, it is known that Out$(F_n)$ for $n \geq 2$ is acylindrically hyperbolic \cite{acylofouterspace}. But none of these cover Out$(\mbox{BS}(p,q))$ as $\mbox{BS}(p,q)$ is neither hyperbolic nor infinite-ended. In \cite{AutomorphismBSgroups}, Clay computed the presentation of Out$(\mbox{BS}(p,q))$ when $p$ properly divides $q$. Further, Clay showed that Out$(\mbox{BS}(p,q))$ admits two natural graph of groups structures $(\Gamma_1,\mathcal{G}_1)$ and $(\Gamma_2,\mathcal{G}_2)$, where the underlying graph $\Gamma_1$ is an infinite ray and the underlying graph $\Gamma_2$ is a single edge. By Bass-Serre theory, these two graph of groups structures correspond to two actions of Out$(\mbox{BS}(p,q))$ on different Bass-Serre trees $X_1$ and $X_2$ respectively. We examine both actions in detail and compare them. The following theorem summarizes the distinctions between these two tree actions.

\begin{customthm}{E}\label{thmE}\textit{Let $G=Out(BS(p,q))$ where $q=pn$ for some $p,|n|>1$. Let $X_1$ be the Bass-Serre tree of the infinite ray of groups $(\Gamma_1,\mathcal{G}_1)$ of $G$ and let $X_2$ be the Bass-Serre tree of the edge of groups $(\Gamma,\mathcal{G})$ of $G$.  Then we have the following results:
\begin{enumerate}
    \item $G \curvearrowright X_1$ is not an acylindrical action while $G \curvearrowright X_2$ is non-elementary acylindrical, thus $G \curvearrowright X_1$ is not equivalent to $G \curvearrowright X_2$ ;
    \item $G \curvearrowright X_1$ and $G \curvearrowright X_2$ are both non-elementary WPD actions;
    \item $G \curvearrowright X_1$ and $G \curvearrowright X_2$ have the same set of hyperbolic elements;
    \item $G \curvearrowright X_2$ is the largest acylindrical hyperbolic action.  
\end{enumerate}}
\end{customthm}

\textbf{Outline of the paper.} In the first part of the paper, we establish acylindricity results for general graphs of groups. In Section \ref{section: BS theory}, we recall the basic notions about Bass-Serre theory. In Section \ref{section: acyl}, we recall the definitions and results about acylindrical actions. In Section \ref{section: quotient action}, we examine quotient actions of groups acting on trees, provide a proof of Theorem \ref{thmA}, and derive Corollaries \ref{coroB} and \ref{coroC}. This leads to the conclusion that the quotient group by normal subgroups generated by an equivariant family of subgroups is acylindrically hyperbolic. Section \ref{section: largest section} discusses the largest acylindrical action and aims to prove another main result of the paper, Theorem \ref{thmD}. 

Next we specifically focuses on Out$(\mbox{BS}(p,q))$ in Section \ref{section: OutBS}. We briefly describe the two graph of groups structures for Out$(\mbox{BS}(p,q))$ from Clay's work in \cite{AutomorphismBSgroups}, along with their corresponding actions on trees, and compare these actions. We demonstrate that Out$(\mbox{BS}(p,q))$ has a non-elementary acylindrical action on a tree, confirming that it is an acylindrically hyperbolic group. As an application of the results from Sections \ref{section: quotient action} and \ref{section: largest section}, we show that Out$(\mbox{BS}(p,q))$ has a largest acylindrical action, and that acylindricity is preserved in quotient actions by certain normal subgroups. We illustrate the quotients of Out$(\mbox{BS}(p,q))$ through Example \ref{example: infinite quotient} and Example \ref{example: finite quotient}. In addition, we prove that the two actions are both non-elementary WPD, providing an alternative proof of the acylindrical hyperbolicity of Out$(\mbox{BS}(p,q))$ when $p$ divides $q$ properly. Finally, at the end of Section \ref{section: OutBS}, we prove our last main result, Theorem \ref{thmE}.

\vspace{0.1in}

\textbf{Acknowledgments.}  
We would like to thank our advisor Johanna Mangahas for her exceptional guidance and encouragement. We would also like to thank Carolyn Abbott, Sahana Balasubramanya, Indira Chatterji, Ashot Minasyan and Wenyuan Yang for helpful conversations, and Nicholas Touikan for carefully reading and commenting on an earlier version of this paper. The first author acknowledges travel support from the Simons
Foundation (965204, JM).
\end{section}

\section{Bass-Serre Theory}\label{section: BS theory}

In this section we recall some basic notions of the Bass-Serre theory. There are two equivalent versions of this theory; details can be found in \cite{tree} and \cite{Bass}. We will primarily work with \cite{tree}. However we will occasionally use tools from \cite{Bass} to obtain more clear and easier proofs.

\subsection{Graph of groups}

\begin{definition}\cite[Definition 2.1.1]{tree}
A \textit{graph} $\Gamma$ consists of a set $V(\Gamma)$, a set $E(\Gamma)$ and two maps
$$E(\Gamma)\rightarrow V(\Gamma)\times V(\Gamma),\ \ \ \ e\mapsto (o(e),t(e))$$
and
$$E(\Gamma)\rightarrow E(\Gamma), \ \ \ \  e\mapsto \overline{e}$$
which satisfy the following condition: for each $e\in E(\Gamma)$ we have $\overline{\overline{e}}=e$, $\overline{e}\neq e$ and $o(e)=t(\overline{e})$. An element $v\in V(\Gamma)$ is called a \textit{vertex} of $\Gamma$; An element $e\in E(\Gamma)$ is called an \textit{oriented edge} of $\Gamma$, and $\overline{e}$ is called the \textit{inverse} edge of $e$. The vertex $o(e)=t(\overline{e})$ is called the \textit{origin} of $e$, and the vertex $t(e)=o(\overline{e})$ is called the \textit{terminus} of $e$. These two vertices are called the \textit{extremities} of $e$.
\end{definition}

\begin{figure}[ht]
    \centering
\begin{tikzpicture}
      \draw (0,2) -- (2.5,2) node[midway, sloped, above] {};
      \draw[->, thick] (1.27,2) -- (1.29,2);

      \node[label={below, yshift=-0.3cm:}] at (1.25,2.4) {$e$};
      \node[label={below, yshift=-0.3cm:}] at (0,1.6) {$o(e)$};
      \node[label={below, yshift=-0.3cm:}] at (2.5,1.6) {$t(e)$};
      \node[label={below, yshift=-0.3cm:}] at (8.5,2) {$f$};
      \node[label={below, yshift=-0.3cm:}] at (5.5,2.5) {$t(f)$};
      \node[label={below, yshift=-0.3cm:}] at (5.5,2) {\rotatebox{90}{$=$}};
      \node[label={below, yshift=-0.3cm:}] at (5.5,1.5) {$o(f)$};

      \tikzset{enclosed/.style={draw, circle, inner sep=0pt, minimum size=.1cm, fill=black}}
      
      \node[enclosed, label={right, yshift=.2cm:}] at (0,2) {};
      \node[enclosed, label={right, yshift=.2cm:}] at (2.5,2) {};
      
      \node[enclosed, label={right, yshift=.2cm:}] at (6,2) {};

      \draw(6, 2) arc(-180:180:1);
      \draw[->, thick] (8,1.98) -- (8,2.02);
\end{tikzpicture}

    \caption{Two examples of graphs. The left graph consists of one edge with different extremities, whereas the right graph consists of one edge with the same extremities.}
    \label{fig: graphs}
\end{figure}
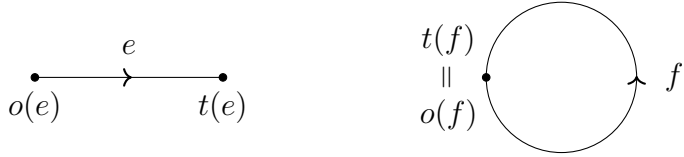

\begin{definition}
 A \textit{graph of groups} $(\Gamma,\mathcal{G})$ consists of a connected graph $\Gamma$, and a family of groups $\mathcal{G}$, together with:
\begin{enumerate}
\item a \textit{vertex group} $G_v\in\mathcal{G}$ for each $v\in V(\Gamma)$;
\item an \textit{edge group} $G_e\in \mathcal{G}$ for each $e\in E(\Gamma)$, such that $G_e=G_{\overline{e}}$; and
\item a monomorphism $\alpha_e: G_e\hookrightarrow G_{t(e)}$ for each $e\in E(\Gamma)$.
\end{enumerate}
Further, we say $(\Gamma,\mathcal{G})$ is a \textit{finite} graph of groups if the graph $\Gamma$ is both locally finite and has finite diameter.
\end{definition}
There are two particular cases of graph of groups, which we will describe in Example \ref{example:edge of groups} and Example \ref{example:loop of groups}. A more general graph of groups can be seen as a combination of these following two types.

\begin{example}\label{example:edge of groups} An \textit{edge of groups} is pictured as follows. 

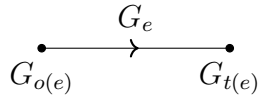
\begin{center}
\begin{tikzpicture}
      \draw (0,2) -- (2.5,2) node[midway, sloped, above] {};
      \draw[->, thick] (1.27,2) -- (1.29,2);

      \node[label={below, yshift=-0.3cm:}] at (0,2.8) {};
      \node[label={below, yshift=-0.3cm:}] at (1.25,2.4) {$G_e$};
      \node[label={below, yshift=-0.3cm:}] at (0,1.6) {$G_{o(e)}$};
      \node[label={below, yshift=-0.3cm:}] at (2.5,1.6) {$G_{t(e)}$};
      
      \tikzset{enclosed/.style={draw, circle, inner sep=0pt, minimum size=.1cm, fill=black}}
      
      \node[enclosed, label={right, yshift=.2cm:}] at (0,2) {};
      \node[enclosed, label={right, yshift=.2cm:}] at (2.5,2) {};
\end{tikzpicture}
\end{center}

The monomorphisms are $\alpha_e: G_e\hookrightarrow G_{t(e)} $ and $\alpha_{\overline{e}}: G_e\hookrightarrow G_{o(e)}$.
\end{example}

\begin{example}\label{example:loop of groups}
A \textit{loop of groups} is pictured as follows.

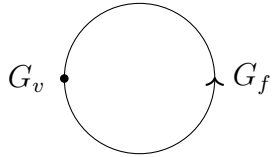
\begin{center}
\begin{tikzpicture}
      \node[label={below, yshift=-0.3cm:}] at (8.5,3.2) {};
      \node[label={below, yshift=-0.3cm:}] at (8.5,2) {$G_f$};
      \node[label={below, yshift=-0.3cm:}] at (5.5,2) {$G_v$};

      \tikzset{enclosed/.style={draw, circle, inner sep=0pt, minimum size=.1cm, fill=black}}

      \node[enclosed, label={right, yshift=.2cm:}] at (6,2) {};

      \draw(6, 2) arc(-180:180:1);
      \draw[->, thick] (8,1.98) -- (8,2.02);
\end{tikzpicture}
\end{center}

The monomorphisms are $\alpha_f: G_f\hookrightarrow G_v $ and $\alpha_{\overline{f}}: G_f\hookrightarrow G_v$.

\end{example}

\begin{definition}
Let $(\Gamma,\mathcal{G})$ be a graph of groups. Consider the group $F(\Gamma,\mathcal{G})$ generated by the vertex groups $G_v$ for $v\in V(\Gamma)$ and the elements $e$ of $E(\Gamma)$, subject to the following relations: 
$$\overline{e}=e^{-1}\ \ \ \ \mbox{and}\ \ \ \ e\alpha_e(g)e^{-1}=\alpha_{\overline{e}}(g) \ \ \ \ \mbox{for any}\ e\in E(\Gamma),\ g\in G_e$$
i.e., the group $F(\Gamma,\mathcal{G})$ is the quotient of the free product of all $G_v$ and edges $e\in E(\Gamma)$ by the normal subgroup generated by elements $e\overline{e}$ and $e\alpha_e(g)e^{-1}\alpha_{\overline{e}}(g)^{-1}$.
\end{definition}

\begin{definition}\cite[Definition 5.1.9]{tree} Let $c$ be a path in $\Gamma$. We let $e_1,...,e_n$ denote the edges of $c$ where $n$ is the length of $c$, and put $v_i=o(e_{i+1})=t(e_i)$.
A \textit{word of type $c$} is an element in $F(\Gamma,\mathcal{G})$ of the form
$$g_0e_1g_1e_2...e_ng_n$$
where $g_i\in G_{v_i}$.
\end{definition}

\subsection{Fundamental group of graph of groups}

\begin{definition} Two equivalent definitions of the fundamental group of a graph of groups:
\begin{enumerate}
    \item[(a)] Let $(\Gamma,\mathcal{G})$ be a graph of groups, and let $v_0$ be a vertex of $\Gamma$. The \textit{fundamental group of $(\Gamma,\mathcal{G})$ at $v_0$}, denoted by $\pi_1(\Gamma, v_0)$, is the subgroup of $F(\mathcal{G}, \Gamma)$ consists of words whose type is a circuit based at $v_0$ (i.e. the path $c=(e_1,...,e_n)$ satisfies $o(e_1)=t(e_n)=v_0$).

\vspace{0.075in}
    
    \item[(b)] Let $T$ be a maximal tree of $\Gamma$. By $E(\Gamma\setminus T)$, we denote the set of edges in $E(\Gamma)$ but not in $E(T)$. The \textit{fundamental group of $(\Gamma, \mathcal{G})$ at $T$}, denoted by $\pi_1(\Gamma, T)$, is the group generated by the vertex groups $G_v$ for $v\in V(\Gamma)$ and  edges $e$ of $E(\Gamma\setminus T)$, subject to the following relations:
    $$\overline{e}=e^{-1}\ \ \ \ \mbox{and}\ \ \ \ e\alpha_e(g)e^{-1}=\alpha_{\overline{e}}(g) \ \ \ \ \mbox{for any}\ e\in E(\Gamma\setminus T),\ g\in G_e$$
    In other words, it is the quotient of $F(\Gamma,\mathcal{G})$ by the normal subgroup generated by the set of edges in $T$. Each geometric edge in $\Gamma$ but not in $T$ is called a \textit{stable letter} of $\pi_1(\Gamma, T)$. We note that the number of stable letters is independent of the choice of $T$. This number is called the \textit{first Betti number} of $\Gamma$, and we denote it by $b_1(\Gamma)$ (i.e., $b_1(\Gamma)=\frac{1}{2}|E(\Gamma\setminus T)|$).
\end{enumerate}
\end{definition}
 
We now describe the fundamental groups of the two particular cases of graph of groups that mentioned in Example \ref{example:edge of groups} and Example \ref{example:loop of groups}.

\begin{example}\label{example:free amalgamation}
In this example, we give a presentation of the fundamental group of an edge of groups by using definition (b) above. As pictured in Example \ref{example:edge of groups}, the maximal subtree $T$ of an edge $e$ is the edge itself. So the presentation of the fundamental group is given by
$$\pi_1(\Gamma, T)=\langle G_{o(e)}, G_{t(e)}\hspace{0.1cm} |\hspace{0.1cm}  \alpha_{e}(g)=\alpha_{\overline{e}}(g)\hspace{0.1cm} \mbox{for any}\ g\in G_e\rangle$$
In fact, this fundamental group is the free amalgamated product of the two vertex groups over the edge group, i.e. $\pi_1(\Gamma, T)= G_{o(e)}\ast_{G_e} G_{t(e)}$
\end{example}

\begin{example}\label{example:HNN extension}
Using definition (b) above, we now give a presentation of the fundamental group of a loop of groups. As pictured in Example \ref{example:loop of groups}, the maximal subtree $T$ of a loop $f$ is the vertex $v$. So the presentation of the fundamental group is given by
$$\pi_1(\Gamma, T)=\langle G_v, f\hspace{0.1cm}|\hspace{0.1cm}\overline{f}=f^{-1},\hspace{0.1cm}f\alpha_{e}(g)f^{-1}=\alpha_{\overline{e}}(g)\hspace{0.1cm}\mbox{for any}\ g\in G_f\rangle$$
In fact, this fundamental group is an HNN-extension $G_v\ast_{\alpha_e(G_f)\sim \alpha_{\overline{e}}(G_f)}$ where $f$ is a stable letter.
\end{example}

The following proposition shows that the two definitions above do not depend on which base vertex $v_0$ or maximal subtree $T$ we chose. Moreover, there is a natural projection map $F(\Gamma, \mathcal{G})\rightarrow \pi_1(\Gamma, T)$ given by mapping edges in the maximal tree $T$ to the identity.

\begin{proposition}\cite[Proposition 5.1.20]{tree}
Let $(\Gamma, \mathcal{G})$ be a graph of groups, let $v_0\in V(\Gamma)$ and let $T$ be a maximal tree of $\Gamma$. The natural projection map $F(\Gamma, \mathcal{G})\rightarrow \pi_1(\Gamma, T)$ induces an isomorphism of $\pi_1(\Gamma, v_0)$ onto $\pi_1(\Gamma, T)$.
\end{proposition}

Therefore, for the rest of this paper, we will sometimes use $\pi_1(\Gamma,\mathcal{G})$ to denote the fundamental group of $(\Gamma,\mathcal{G})$.

\subsection{Bass-Serre tree}\label{subsection: construction of BS tree}
Let $(\Gamma, \mathcal{G})$ be a graph of groups, and let $T$ be a maximal tree of $\Gamma$. Set $G=\pi_1(\Gamma, T)$. We can construct a graph $X$ that satisfies the following properties:
\begin{enumerate}
    \item[(a)] there is an action of $G$ on $X$ by left multiplication;
    \item[(b)] the quotient $X/G$ induced by the action of $G$ on $X$ is isomorphic to $\Gamma$; 
    \item[(c)] if $\Tilde{x}$ is a vertex or an edge of $X$, and $\Tilde{y}=g\Tilde{x}$ for some $g\in G$. Then the stabilizer of $\Tilde{y}$ in $G$ is $gG_{q(\Tilde{x})}g^{-1}$, where $q: X\rightarrow \Gamma$ is the quotient map.
    \item[(d)] there exist section maps  $V(\Gamma)\rightarrow V(X)$ and $E(\Gamma)\rightarrow E(X)$ of the quotient map $q$.
\end{enumerate}
In fact, these conditions force the graph $X$ to be as follows: the vertex set $V(X)$ is the disjoint union of left cosets of vertex groups, i.e.
$$V(X)=\bigsqcup_{v\in V(\Gamma)} G/G_{v}$$
Likewise, the edge set $E(X)$ is the disjoint union of left cosets of edge groups, i.e.
$$E(X)=\bigsqcup_{e\in E(\Gamma)} G/G_{e}$$

The graph $X$ is defined by attaching the edges to the vertices in the following way:
$$o(gG_e)=gG_{o(e)}\ \ \ \ \mbox{and}\ \ \ \ t(gG_e)=geG_{t(e)}$$

We remark that $e$ is the identity if $e$ is in the maximal tree $T$. From the construction above, we can see that the quotient $X/G$ has a graph of groups structure that is isomorphic to $(\Gamma, \mathcal{G})$. We conclude this subsection with the following two theorems, known as the fundamental theorems of Bass-Serre theory.

\begin{theorem}\cite[Theorem 5.3.12]{tree}
Let $(\Gamma, \mathcal{G})$ be a graph of groups, and let $T$ be a maximal tree of $\Gamma$. Set $G=\pi_1(\Gamma, T)$. Then the graph $X$ constructed above is a tree, and we call it the \textit{Bass-Serre tree} associated to $(\Gamma, \mathcal{G})$. In particular, $X$ is the universal cover of the quotient graph $\Gamma$.
\end{theorem}

\begin{theorem}\cite[Theorem 5.3.13]{tree}
Let $G$ be a group which acts on a tree $X$ without edge inversions (i.e. $ge\neq \bar{e}$ for any $g\in G$ and $e\in E(X)$). Then $G$ can be identified with the fundamental group of a certain graph of groups $(\Gamma, \mathcal{G})$, where $\Gamma$ is the quotient graph $X/G$. Moreover, there is an $G$-equivariant isomorphism of the Bass-Serre tree associated to $(\Gamma,\mathcal{G})$ to $X$.
\end{theorem}

\subsection{Action on the Bass-Serre tree}  
\begin{definition}\cite[Definition 2.4]{Foresterdeformation}
Let $G$ be a group acting on a tree $X$. An element $\gamma\in G$ is \textit{elliptic} if it has a fixed point, and \textit{hyperbolic} otherwise. We define the \textit{length function} of $X$ by 
$$\ell_{X}(\gamma)=\underset{x\in V(X)}{\mbox{min}}d(x,\gamma x)$$
Thus $\ell_X(\gamma)=0$ if and only if $\gamma$ is elliptic. If $\gamma$ is hyperbolic, then the set
$$\{x\in X\hspace{0.1cm}| \hspace{0.1cm}d(x,\gamma x)=\ell_X(\gamma)\}$$
is a $\gamma$-invariant linear subtree, called the \textit{axis} of $\gamma$. We denote it by $\mbox{Axis}(\gamma)$. The action of $\gamma$ on $\mbox{Axis}(\gamma)$ is by a translation of amplitude $\ell_X(\gamma)$.
\end{definition}

\begin{definition}\cite[Definition 2.10]{acylactionontrees}
Let $G$ be a group acting on a tree $X$. We say $G$ is acting \textit{elliptically}, or that $G$ is \textit{elliptic}, if the action of $G$ on $X$ has a fixed point.
\end{definition}

\section{Acylindrical tree actions}\label{section: acyl}
In this section, we recall definitions and results about acylindrical actions that we will require throughout the paper. For more details we refer to \cite{Osin}.
   
\begin{definition}\label{acyldef}
Let $G$ be a group acting on a hyperbolic space $S$ by isometry. The action is called \textit{acylindrical} if for every $\varepsilon > 0$ there exist $R, N > 0$ such that for every two points $x, y$ in $S$ with $d(x, y) \geq R$, the set
\[\{g \in G\hspace{0.1cm}|\hspace{0.1cm} d(x, gx) \leq \varepsilon,\hspace{0.1cm}  d(y, gy) \leq \varepsilon \}\]
contains at most $N$ elements.
\end{definition}

When we consider acylindrical actions on trees in particular, we use the definition \ref{acyltreeactdefn} below, but we note that there is at least one other commonly used defition for $(k,C)$-acylindrical actions on trees (see \cite{Selaacyl}, \cite{Weidmannkcacylindrical}, \cite{JSJdecomposition}), where the parameter $k$ differs by $1$ from the following definition. 

\begin{definition}\label{acyltreeactdefn}
Let $G$ be a group acting on a simplicial tree $T$ and let $k \geq 0$ and $C > 0$ be integers. We say that the action is $(k, C)$-\textit{acylindrical} if the pointwise stabilizer of any edge path in $T$ of length at least $k$ contains at most $C$ elements.
\end{definition}

The following theorem is due to Minasyan and Osin \cite{acylactionontrees}, which we include without proof. This theorem shows that the above two definitions are, in fact, equivalent.

\begin{theorem}\cite[Lemma 4.2]{acylactionontrees}\label{thm: acyl=kcacyl}
Let $G$ be a group acting on a simplicial tree $T$ by isometries. This action is acylindrical if and only if there exist constants $k \geq 0$ and $C \geq 1$ such that the action of $G$ on $T$ is $(k, C)$-acylindrical.
\end{theorem}

\begin{definition}\cite[Definition 2.4]{Osin}
Let $G$ be a group acting on isometrically on a hyperbolic space $S$. An element $g\in G$ is called \textit{elliptic} if all orbits of $g$ are bounded; and \textit{hyperbolic} (or \textit{loxodromic}) if the map $\mathbb{Z}\rightarrow S$ defined by $n\mapsto g^ns$ is a quasi-isometry for some (equivalently, any) $s\in S$. Two hyperbolic elements $g,h\in G$ are called \textit{independent} if the sets $\{g^{\pm \infty}\}$ and $\{h^{\pm \infty}\}$ are disjoint.
\end{definition}

\begin{remark}
Let $G$ be a group acting on a tree $T$. It is known \cite[Proposition I.2.10 and Corollary]{tree} that the action of $G$ acting on $T$ has bounded orbits if and only if there exists some point $x\in T$ fixed by the action $G\curvearrowright T$. Therefore we can remove any ambiguity as to the definition of an elliptic action in the context of simplicial trees.
\end{remark}

The following theorem classifies acylindrical actions on hyperbolic spaces.

\begin{theorem}\cite[Theorem 1.1]{acylhypgps}\label{acylhypgrpclass}
Let $G$ be a group acting acylindrically on a hyperbolic space. Then $G$ satisfies exactly one of the following three conditions.
\begin{enumerate}
\item $G$ has bounded orbits.
\item $G$ is virtually cyclic and contains a hyperbolic element.
\item $G$ contains infinitely many independent hyperbolic elements.
\end{enumerate}
\end{theorem}

As an application of the above theorem, every element of a group acting acylindrically on a hyperbolic space
is either elliptic or hyperbolic. In the theorem above, $(1)$ and $(2)$ are called \textit{elementary} actions, whereas $(3)$ is called a \textit{non-elementary} action.

\begin{definition}
We say a group $G$ is \textit{acylindrically hyperbolic} if it admits a non-elementary acylindrical action on a hyperbolic metric space $S$ by isometries.
\end{definition}

\section{Preserving acylindricity under quotient}\label{section: quotient action}
In this section, we provide sufficient conditions on normal subgroups of $G$ such that acylindricity of the action of $G$ on a tree $T$ is preserved under quotients by such normal subgroups. We begin by recalling the definition of an equivariant family of subgroups of $G$.

\begin{definition}
Let $G$ be a group acting on a tree $T$. For each vertex $v\in V(T)$, let $R_v$ be a subgroup of the stabilizer of $v$ in $G$. We say the family of subgroups $\{R_v\}$ is an \textit{equivariant family} of subgroups of $G$ if it satisfies the following condition:
\begin{itemize}
    \item[-] If $g$ lies in $G$ and $v$ is a vertex of $T$ then $gR_vg^{-1}=R_{gv}$.
\end{itemize}
\end{definition}

\begin{remark}\label{rmk: orbitrep of equi-family}
The equivariance condition implies that for each vertex $v$ the subgroup $R_v$ is normal in the stabilizer of $v$, and that the subgroup $\langle R_v \rangle$ of $G$ generated by the family is normal in $G$. Further, if we choose orbit representatives $v_i$ for the action of $G$ on the vertices of $T$, then $\langle R_v \rangle$ is the normal closure of the set $\{ R_{v_i} \}$. Letting $(\Gamma,\mathcal{G})$ be the graph of groups associated to the action of $G$ on $T$, we can always find such a set of orbit representatives by choosing a lift $\tilde{\Gamma}$ of $\Gamma$ in $T$, namely $\langle R_v \rangle=\langle\langle R_{v_i}| v_i\in V(\tilde{\Gamma})\rangle\rangle$. Moreover, moving forward we will choose a particular lift $\tilde{\Gamma}$ of $\Gamma$ such that the vertex stabilizers in $\tilde{\Gamma}$ remain the same as the vertex groups in $\Gamma$, i.e., $G_{\hat{v}} = G_v$ for lifts of vertices $\hat{v}$ in $\Gamma$ to $v$ in $\tilde{\Gamma}$. With this choice of lift, we identify $R_v$'s as a subgroup of the vertex groups $G_{\hat{v}}$ in $\Gamma$ for each $v \in V(\tilde\Gamma)$. This implies that $\langle R_v \rangle=\langle\langle R_{\hat{v}}|\hat{v}\in V(\Gamma)\rangle\rangle$.

\end{remark}

Let us consider an action on a tree. Our next result classifies normal subgroup quotients which preserve a tree structure.

\begin{proposition}\label{prop: quotient_tree}
Let $G$ be a group acting on a simplicial tree $T$, and let $N$ be a normal subgroup of $G$. The quotient graph $T/N$ is a tree if and only if $N$ is generated by an equivariant family of subgroups $\{R_v\}$, i.e., $N = \langle R_v \rangle$.
\end{proposition}

\begin{proof}
Let $\Gamma_{N}$ denote the quotient graph $T/N$. Consider the action of $N \curvearrowright T$ that we obtained by restricting the action of $G \curvearrowright T$ to $N$. 
Assume that $\Gamma_N$ is a tree. We note that $\Gamma_{N}$ can be lifted to a subtree $\tilde{\Gamma}_{N}$ in $T$ such that $N\cdot \tilde{\Gamma}_N=T$, i.e., $\tilde{\Gamma}_N$ is a fundamental domain of $T$ under the action $N\curvearrowright T$. By Bass-Serre theory, $N$ admits a graph of groups structure $(\Gamma_N, \mathcal{G}_N)$ such that for each vertex $\bar{v}$ in $\Gamma_N$, the vertex group is $G_{\bar{v}}=\mbox{Stab}_{N}(v)$ where $v$ is the lift of $\bar{v}$ in $\tilde{\Gamma}_N$, and for each edge $\overline{e}$ in $\Gamma_N$, the edge group is $G_{\overline{e}}=\mbox{Stab}_N(e)$ where $e$ is the lift of $\bar{e}$ in $\tilde{\Gamma}_N$. Further, we have $N=\pi_1(\Gamma_N, \mathcal{G}_N)$. Since $\Gamma_N$ is a tree, then $N=\langle \sqcup_{\overline{v}\in V(\Gamma_N)} G_{\overline{v}}\rangle=\langle\sqcup_{v\in V(\tilde{\Gamma}_N)}\mbox{Stab}_N(v)\rangle$. Now let $R_v=\mbox{Stab}_N(v) = \mbox{Stab}_G(v) \cap N \leq \mbox{Stab}_G(v)$ for each vertex $v$ in $\tilde{\Gamma}_N$. Since $N\cdot \tilde{\Gamma}_N= T$, we can repeat the choice equivariantly throughout vertices in $T$, i.e., for any $w\in V(T)$, we have $w=h\cdot v$ for some $h\in N$ and $v\in V(\tilde{\Gamma}_N)$, we then choose $R_{w}=hR_vh^{-1}$. It follows that $\{R_v\}$ is an equivariant family of subgroups of $G$. Under this choice of $\{R_v\}$, it is clear that $N=\langle R_v\rangle$. Conversely, suppose $N$ is a normal subgroup generated by $\{R_v\}$, an equivariant family of subgroups of $G$. Since each $R_v$ is a subgroup of some vertex stabilizer, it does not contain any stable letters. In addition, quotient $T$ by each $R_v$ folds edges emanating from $v$. Thus, it follows from Bass-Serre theory and the definition of equivariant family that in the quotient graph, valence of each vertex might be reduced but no loop is created. Therefore the quotient graph $T/N$ remains a tree.

\end{proof}

\begin{theorem}\label{thm: quotientActionAcyl}
Let $G$ be a group with a $(1,C)$-acylindrical action on a simplicial tree $T$, i.e. the edge stabilizers are finite and let $\{R_v\}$ be an equivariant family of subgroups of $G$. Then the quotient action $G/\langle R_v \rangle$ on $T/ \langle R_v \rangle$ remains $(1,C)$-acylindrical.
\end{theorem}

\begin{proof}
For simplicity, we denote $N=\langle R_v\rangle$, denote the quotient group $G/N$ by $\overline{G}$, and denote the quotient graph $T/N$ by $\Gamma_N$. Since $\Gamma_N$ is a tree, then $\Gamma_N$ can be lifted to a subtree $\tilde{\Gamma}_N$ (an isometrically embedded copy of $\Gamma_N$) in $T$ such that $N\cdot \tilde{\Gamma}_N=T$, i.e., $\tilde{\Gamma}_N$ is the fundamental domain of $T$ under the action $N\curvearrowright T$.  We now assume that $G\curvearrowright T$ is $(1,C)$-acylindrical. Then every edge stabilizer under the action $G\curvearrowright T$ is finite. Let $\bar{e} =[\bar{u},\bar{v}]$ be an edge in $\Gamma_N$. Further, let $u$ and $v$ be the lifts of $\bar{u}$ and $\bar{v}$ respectively in $\tilde{\Gamma}_N$. By the construction of $\tilde{\Gamma}_N$, we have $d_T(u,v) = 1$. Choose coset representatives $\Sigma_{N}\subset G$, we can represent elements in $\overline{G}$ by elements in $\Sigma_{N}$. For each $g\in \mbox{PStab}_{\overline{G}}([\bar{u},\bar{v}])=\mbox{PStab}_{\overline{G}}(\{\bar{u},\bar{v}\})$, we claim that $g\in \mbox{PStab}_G([u,v])$. If not, then $g\cdot[u,v] \neq n\cdot[u,v]$ for any $n \in N$. However, since $g$ stabilizes $\bar{u}$ and $\bar{v}$, then $g\cdot u=n\cdot u$ and $g\cdot v=n'\cdot v$ for some $n,n'\in N$ and $n\neq n'$. This implies that the path $[nv,nu]\cup [nu,n'v]$ in $T$ creates a circuit in the quotient $\Gamma_N$, contradicting the fact that $\Gamma_N$ is a tree (Proposition \ref{prop: quotient_tree}). As we have $|\mbox{PStab}_G([u,v])|\leq C$. This shows that $|\mbox{PStab}_{\overline{G}}([\bar{u},\bar{v}])|\leq C$, i.e. each edge stabilizer in the quotient action of $\overline{G} \curvearrowright \Gamma_N$ is finite. Thus the quotient action $\overline{G}\curvearrowright \Gamma_N$ is also $(1,C)$-acylindrical.

\end{proof}

\begin{remark}
Note that the proof technique above for Theorem \ref{thm: quotientActionAcyl} does not work for $(k,C)$-acylindrical actions of $G$ on $T$ for $k \geq 2$. In particular, $g\in \mbox{PStab}_{\overline{G}}([\bar{u},\bar{v}])$ does not imply $g\in \mbox{PStab}_G([u,v])$ in general, as the existence of a single element $n\in N$ such that $g\cdot [u,v] = n\cdot [u,v]$ cannot be guaranteed. The following example is provided by Nicholas Touikan: 

Consider the HNN-extension $G = \langle a,b \rangle *_{\langle a \rangle\sim_{\phi}  \langle b \rangle}$ where $\phi$ is given by $\phi: a\mapsto b$, i.e., $G = \langle a,b,t \  | \  tat^{-1} = b \rangle$. Then the action of $G$ on its Bass-Serre tree $T$ is $(2,1)$-acylindrical. Let us consider the equivariant family of subgroups generated by $R_{\langle a,b \rangle} = \langle ab^{-1} \rangle$. Then we have $N = \langle\langle ab^{-1} \rangle\rangle$ and $\overline{G} = G/N = \langle a,t \ | \  tat^{-1}=a \rangle$. In addition, the quotient tree $\Gamma_N$ is a bi-infinite line pointwise stabilized by the infinite cyclic group $\langle a \rangle$, therefore loosing acylindricity in the quotient action. Now we consider the length two path $[\overline{t^{-1}\langle a,b\rangle}, \overline{at\langle a,b \rangle}]=[t^{-1}\langle a \rangle, t\langle a \rangle]$ in $\Gamma_N$. Then $a \in \Sigma_N$ stabilizes this path. Let $\tilde\Gamma_N = \mbox{Axis}(t)$ in the action of $G\curvearrowright T$, we can lift the path $[\overline{t^{-1}\langle a,b \rangle}, \overline{at\langle a,b \rangle}]$ to $[t^{-1}\langle a,b \rangle, t\langle a,b \rangle]$ in $T$. Then there is no $n \in N$ such that $a \cdot [t^{-1}\langle a,b \rangle, t\langle a,b \rangle] = n \cdot [t^{-1}\langle a,b \rangle, at\langle a,b \rangle]$. Therefore, $a \notin \mbox{PStab}_G([t^{-1}\langle a,b \rangle, t\langle a,b \rangle])$.
\end{remark}

\underline{Notations.} Let $G$ be a group acting on a tree $T$, and let $(\Gamma,\mathcal{G})$ be a finite graph of groups associated to the action $G\curvearrowright T$. Further, let $N$ be the normal subgroup of $G$ generated by an equivariant family of subgroups $\{R_v\}$. We denote the quotient group $G/N$ by $\overline{G}$, the quotient tree $T/N$ by $\overline{T}$. Choose a coset representatives $\Sigma_{N}\subset G$, we then identify elements in $\overline{G}$ by elements in $\Sigma_{N}$. We use $v$ to denote vertex in $T$, $\hat{v}$ to denote vertex in $\Gamma$ and $\bar{v}$ to denote vertex in $\overline{T}$. For any two vertices $\hat{u}$ and $\hat{v}$ in $\Gamma$, we denote $[\hat{v},\hat{v}]$ the geometric geodesic path (disregrading the orientation) joining $\hat{u}$ and $\hat{v}$.

We need a few lemmas before proving the main theorem, Theorem \ref{thm: quotientActionAcyl}, of this section.

\begin{lemma}\label{lemma: valence in quotient}
Let $v$ be a vertex in $T$. The vertex $\bar{v}$ in $\overline{T}$ has valence at least two if and only if at least one the following holds:
\begin{enumerate}
\item there is an edge $\hat{e}$ in $\Gamma$ emanating from $\hat{v}$ such that $G_{\hat{v}}\neq \langle G_{\hat{e}}, R_{\hat{v}}\rangle$; 
\item $\hat{v}$ has valence at least two.
\end{enumerate}
\end{lemma}
\begin{proof}
Assume that condition (1) holds. As in Remark \ref{rmk: orbitrep of equi-family}, we can assume Stab$(v) = G_{\hat{v}}$. Let $e$ be the lift (in $\tilde\Gamma$) of $\hat{e}$ that emanating from $v$. Since $G_{\hat{v}}\neq \langle G_{\hat{e}}, R_{\hat{v}}\rangle$, then there exists a non-trivial element $g\in G_{\hat{v}}\setminus \langle G_{\hat{e}}, R_{\hat{v}}\rangle$. Further, since $g\not\in G_{\hat{e}}$, we have $ge$ and $e$ are two different edges emanating from $v$. Since $g\not\in R_{\hat{v}}$. then $gG_{\hat{e}}\neq hG_{\hat{e}}$ for any $h\in R_{\hat{v}}$. This shows that $ge$ and $e$ lie in two different $\langle R_v\rangle$-orbits. Thus $\bar{v}$ has valence at least two. Now assume that condition (2) holds, then there are at least two different $G$-orbits of edges emanating from $v$. It follows that there are at least two different $\langle R_v \rangle$-orbits of edges emanating from $v$. Therefore $\bar{v}$ has valence at least two.

For the only if direction, we assume that $\bar{v}$ has valence at least two. Then there are two different edges $\bar{e}$ and $\bar{f}$ emanating from $\bar{v}$ in $\overline{T}$. Let $e$ (resp. $f$) be an edge emanating from $v$ in $T$ such that $q(e)=\bar{e}$ (resp. $q(f)=\bar{f}$), where $q: T\rightarrow \overline{T}$ is the natural quotient map. There are two cases, depending on whether $e$ and $f$ lie in the same $G$-orbit or not. If $e$ and $f$ do not lie in the same $G$-orbit, then $\hat{v}$ has valence at least two. If $e$ and $f$ lie in the same $G$-orbit $\hat{e}$, without loss of generality, we may assume that Stab$(e) = G_{\hat{e}}$. Then there exists an element $g\in G_{\hat{v}}\setminus G_{\hat{e}}$ such that $f=ge$, i.e., $f$ is represented by $gG_{\hat{e}}$. Further, since $\bar{e}\neq \bar{f}$, then $gG_{\hat{e}}\neq hG_{\hat{e}}$ for any $h\in R_{\hat{v}}$. This implies that $g\in G_{\hat{v}}\setminus \langle G_{\hat{e}}, R_{\hat{v}}\rangle$. Thus $G_{\hat{v}}\neq \langle G_{\hat{e}}, R_{\hat{v}}\rangle$. This completes the proof. 
\end{proof}

\begin{lemma}\label{lemma: quotient infinite tree}
$\overline{T}$ is an infinite tree if at least one of the following holds:
\begin{enumerate}
    \item there are two vertices $\hat{u}$ and $\hat{v}$ in $\Gamma$ such that $\langle G_{\hat{e}},R_{\hat{u}}\rangle\neq G_{\hat{u}}$ and $\langle G_{\hat{f}},R_{\hat{v}}\rangle\neq G_{\hat{v}}$, where $\hat{e}$ and $\hat{f}$ are the edges on the geodesic path $[\hat{u},\hat{v}]$ that adjacent to $\hat{u}$ and $\hat{v}$ respectively.
    \item $\Gamma$ contains a circuit.
\end{enumerate}
\end{lemma}

\begin{proof}

By Proposition \ref{prop: quotient_tree}, $\overline{T}$ is always a tree. Then it suffices to show that $\overline{T}$ is infinite if at least one of the above conditions holds. Assume that condition (1) holds. Let $[\hat{u},\hat{v}]$ be the geodesic segment in $\Gamma$ joining $\hat{u}$ and $\hat{v}$, we can lift $[\hat{u},\hat{v}]$ to a geodesic segment $[u,v]$ inside $\tilde\Gamma$, the partcular lift of $\Gamma$ as in Remark \ref{rmk: orbitrep of equi-family}. Note that $[u,v]$ generates a subtree $T'$ of $T$, i.e., $G\cdot [u,v] =T'\subseteq T$.  Since $G_{\hat{e}}\neq G_{\hat{u}}$ and $G_{\hat{f}}\neq G_{\hat{v}}$, then $T'$ must be an infinite tree. By Lemma \ref{lemma: valence in quotient}, we have both $\bar{u}$ and $\bar{v}$ have valence at least two and the image of $T'$ under the quotient map $q: T\rightarrow \overline{T}$ contains a bi-infinite line. More precisely, there exist $g\in G_{\hat{u}}\setminus \langle G_{\hat{e}},R_{\hat{u}}\rangle$ and $h\in G_{\hat{v}}\setminus \langle G_{\hat{f}},R_{\hat{v}}\rangle$ such that $g,h\in \overline{G}$, $gh$ is hyperbolic in both $G$ and $\overline{G}$, and both $T$ and $\overline{T}$ contain $\mbox{Axis}(gh)$. Therefore $\overline{T}$ is infinite. Now assume that condition (2) holds. This implies that $G$ contains a stable letter $t$. We note that the equivariant family of subgroups $\{R_v\}$ does not contain any stable letters of $G$, then each axis of the stable letters survives under the quotient map $q: T\rightarrow \overline{T}$. Thus $\overline{T}$ contains $\mbox{Axis}(t)$. This implies that $\overline{T}$ is infinite.
\end{proof}

The following proposition provides sufficient conditions for the graph of groups structure $(\Gamma,\mathcal{G})$ of $G$ to ensure that the quotient action contains enough hyperbolic elements.

\begin{proposition}\label{prop: nonelementary}
The quotient action $\overline{G}\curvearrowright \overline{T}$ contains at least two independent hyperbolic elements if $\{R_v\}$ satisfies at least one of the following conditions:
\begin{enumerate}
    \item the first Betti number $b_1(\Gamma)\geq 2$ where $\Gamma=T/G$.
    \item there are two vertices $\hat{u}$ and $\hat{v}$ in $\Gamma$ such that $\langle G_{\hat{e}},R_{\hat{u}}\rangle\neq G_{\hat{u}}$ and $\langle G_{\hat{f}},R_{\hat{v}}\rangle\neq G_{\hat{v}}$, where $\hat{e}$ and $\hat{f}$ are the edges on the geodesic path $[\hat{u},\hat{v}]$ that are adjacent to $\hat{u}$ and $\hat{v}$ respectively. Further, there exist two elements $g,h\in G_{\hat{v}}\setminus \langle G_{\hat{f}}, R_{\hat{v}}\rangle$ such that $gh^{-1}\not\in \langle G_{\hat{f}}, R_{\hat{v}}\rangle$.
    \item $\Gamma$ contains an immersed circuit. Further, there exists a vertex $\hat{v}$ and an edge $\hat{e}$ adjacent to $\hat{v}$ on the circuit such that $\langle G_{\hat{e}}, R_{\hat{v}}\rangle\neq G_{\hat{v}}$.
\end{enumerate}
\end{proposition}

\begin{proof}
Assume that condition (1) holds. Then $G$ contains at least two different stable letters $t_1$ and $t_2$. Since the equivariant family of subgroups $\{R_v\}$ does not contain any stable letters of $G$, then we have $t_1,t_2\in \overline{G}$ and $\overline{T}$ contains $\mbox{Axis}(t_1)$ and $\mbox{Axis}(t_2)$. This shows that $t_1$ and $t_2$ are two independent hyperbolic elements of the action $\overline{G}\curvearrowright \overline{T}$.

Assume that condition (2) holds. First, we can lift the geodesic segment $[\hat{u},\hat{v}]$ to the geodesic segment $[u,v]$ in $\tilde\Gamma$ as in Remark \ref{rmk: orbitrep of equi-family}. Let $e,f$ be the lifts of $\hat{e}$ and $\hat{f}$ on $[u,v]$ that are adjacent to $u$ and $v$ respectively. By assumption, we have $f,gf,hf$ are three different edges emanating from $v$. Moreover, they lie in three different $\langle R_v \rangle$-orbits, i.e., $\bar{v}$ has valence at least three and $q(f),q(gf),q(hf)$ represents three different edges in $\overline{T}$ emanating from $\bar{v}$. Let $g'\in G_{\hat{u}}\setminus \langle G_{\hat{e}}, R_{\hat{u}}\rangle$, similar to the proof of Lemma \ref{lemma: quotient infinite tree}, we have $g'g$ and $g'h$ are two independent hyperbolic elements of both $G$ and $\overline{G}$, and both $T$ and $\overline{T}$ contain $\mbox{Axis}(g'g)$ and $\mbox{Axis}(g'h)$.

Assume that condition (3) holds. Since $\Gamma$ contains an immersed circuit, we have $G$ contains a stable letter $t$. Denote the circuit by $\hat{c}$ and choose a base vertex $\hat{u}$ on $\hat{c}$. Without loss of generality, we may assume that the orientation of $\hat{e}$ is the same as $\hat{c}$, for otherwise we can choose $\hat{c}^{-1}$ instead of $\hat{c}$. Then we can lift $\hat{c}$ to a geodesic segment $[u,tu]$ in $T$ where $q(u)=\hat{u}$, and this geodesic segment generates the axis of $t$ in $T$. More precisely, we have $t^{m} \cdot [u,tu]=\mbox{Axis}(t)$ for all $m\in \mathbb{Z}$. Similar to condition (1), we have $\overline{T}$ contains $\mbox{Axis}(t)$. Let $v,e$ be the lift of $\hat{v},\hat{e}$ in $[u,tu]$ respectively. Since $\langle G_{\hat{e}}, R_{\hat{v}}\rangle\neq G_{\hat{v}}$, similar to the proof of Lemma \ref{lemma: valence in quotient}, there exists $g\in G_{\hat{v}}\setminus \langle G_{\hat{e}}, R_{\hat{v}}\rangle$ such that $ge$ and $e$ are two different edges in $T$ emanating from $v$. We note that $gt$ is a hyperbolic element of the action $G\curvearrowright T$ with translation length $\ell(gt)=\ell(t)$. In fact, the geodesic segment $[gu,gtu]$ generates the axis of $gt$ in $T$. In addition, since $ge$ and $e$ are two different edges in $T$, and $ge$ lies on $[gu,gtu]$. Then $gt$ and $t$ are two independent hyperbolic elements of the action $G\curvearrowright T$. Recall that $g\in G_{\hat{v}}\setminus \langle G_{\hat{e}}, R_{\hat{v}}\rangle$, then $gt,t\in \overline{G}$ and $\overline{T}$ contains both $\mbox{Axis}(gt)$ and $\mbox{Axis}(t)$. This shows that $gt$ and $t$ are two independent hyperbolic elements of the action $\overline{G}\curvearrowright \overline{T}$.

\end{proof}

We now provide a proof of our first main result, which is restated as follows.

\begin{theorem}\label{quotient_Thm}
Let $G$ be a group acting on a simplicial tree $T$, and let $N$ be a normal subgroup of $G$. If $N$ is generated by an equivariant family of subgroups $\{R_v\}$, then the quotient space $T/N$ is a tree. Moreover, suppose the action $G\curvearrowright T$ is a $(1,C)$-acylindrical action. Then the quotient action $G/N\curvearrowright T/N$ is a non-elementary acylindrical tree action if $\{R_v\}$ satisfies at least one of the following conditions:
\begin{enumerate}
    \item the first Betti number $b_1(\Gamma)\geq 2$ where $\Gamma=T/G$.
    \item there are two vertices $\hat{u}$ and $\hat{v}$ in $\Gamma$ such that $\langle G_{\hat{e}},R_{\hat{u}}\rangle\neq G_{\hat{u}}$ and $\langle G_{\hat{f}},R_{\hat{v}}\rangle\neq G_{\hat{v}}$, where $\hat{e}$ and $\hat{f}$ are the edges on the geodesic path $[\hat{u},\hat{v}]$ that are adjacent to $\hat{u}$ and $\hat{v}$ respectively. Further, there exist two elements $g,h\in G_{\hat{v}}\setminus \langle G_{\hat{f}}, R_{\hat{v}}\rangle$ such that $gh^{-1}\not\in \langle G_{\hat{f}}, R_{\hat{v}}\rangle$.
    \item $\Gamma$ contains an immersed circuit. Further, there exists a vertex $\hat{v}$ and an edge $\hat{e}$ adjacent to $\hat{v}$ on the circuit such that $\langle G_{\hat{e}}, R_{\hat{v}}\rangle\neq G_{\hat{v}}$.
\end{enumerate}
\end{theorem}
\begin{proof}
The proof follows directly from Theorem \ref{thm: quotientActionAcyl} and Proposition \ref{prop: nonelementary}.
\end{proof}

We remark that if $(\Gamma,\mathcal{G})$ is finite (i.e. $\Gamma$ has both finite valence and finite diameter), then $G\curvearrowright T$ is a $(1,C)$-acylindrical action if and only if every edge group in $(\Gamma,\mathcal{G})$ is finite. Thus we can obtain the following two particular corollaries as immediate applications of the above theorem.

\begin{corollary}\label{coro: quotient of amalgamation}
Let $G=A\ast_C B$ be a free amalgamated product where $C$ is finite, and let $T$ be its Bass-Serre tree. If $N$ is generated by an equivariant family of subgroups $\{R_v\}$. Then the quotient action $G/N\curvearrowright T/N$ is a non-elementary acylindrical tree action if $\{R_v\}$ satisfies the following:
\begin{enumerate}
    \item $\langle C, R_A\rangle\neq A$ and  $\langle C, R_B\rangle\neq B$; and
    \item there exists two different elements $g$ and $h$ in $ B\setminus \langle C, R_B\rangle$ such that $gh^{-1}\not\in \langle C, R_B\rangle$.
\end{enumerate}
\end{corollary}

\begin{proof}
The proof follows directly from condition (2) of Theorem \ref{quotient_Thm}.
\end{proof}

\begin{corollary}\label{coro: quotient of HNN}
Let $G=A\ast_{H\sim_{\phi} K}$ be an HNN-extension where $H,K$ are finite, and let $T$ be its Bass-Serre tree. If $N$ is generated by an equivariant family of subgroups $\{R_v\}$. Then the quotient action $G/N\curvearrowright T/N$ is a non-elementary acylindrical tree action if $\{R_v\}$ satisfies either $\langle H, R_A \rangle\neq A$ or $\langle K, R_A \rangle\neq A$.
\end{corollary}
\begin{proof}
The proof follows directly from condition (3) of Theorem \ref{quotient_Thm}.
\end{proof}

We will discuss a particular application of Corollary \ref{coro: quotient of amalgamation} to the outer automorphism group of Baumslag-Solitar group in subsection \ref{subsection: quotient action of OutBS}.

\section{Largest acylindrical action of graph of groups}\label{section: largest section}
In this section, our goal is to prove Theorem \ref{thmD}. This theorem provides sufficient conditions for the fundamental group of graph of groups admitting a largest acylindrical action. This is a particular case of Proposition 4.3 in \cite{hyperbolicstructure} which gives sufficient condition for existence of a largest action of a group acting cocompactly on a connected graph. Restricting to action on trees we give a simpler proof using Bass-Serre theory. We begin by briefly recall some definitions from \cite{hyperbolicstructure} defined by Abbott, Balasubramanya and Osin.

\begin{definition}
Let $X$ and $Y$ be two generating sets of group $G$. We say that $X$ is \textit{dominated} by $Y$, denoted by $X \preceq Y$, if 
$$\mbox{sup}_{y \in Y} |y|_X < \infty$$
where $|\cdot |_X = d_X(1, \cdot)$ denotes the word length with respect to $X$. It is clear that $\preceq$ defines a preorder on the generating sets of the group $G$ and induces an equivalence relation:
    \[X \sim Y \Longleftrightarrow X \preceq Y \hspace{0.2cm} \text{and} \hspace{0.2cm} Y \preceq X\]
Let $[X]$ denote the equivalence class of the generating set $X$ and $\mathcal{G}(G)$ be the set of all equivalence classes of generating sets of $G$. The preorder $\preceq$ induces a partial order $\preccurlyeq$ on $\mathcal{G}(G)$ by the rule
\[[X] \preccurlyeq [Y] \Longleftrightarrow X \preceq Y.\]
\end{definition}

Given a group $G$ with a generating set $X$, we denote by $\Gamma(G,X)$ the Cayley graph of $G$ with respect to the generating set $X$.

\begin{definition}
Let $\mathcal{AH}(G)$ be the set of equivalence classes $[X]\in \mathcal{G}(G)$ such that $\Gamma(G,X)$ is hyperbolic and the action $G\curvearrowright \Gamma(G,X)$ is acylindrical. We say an equivalence class of generating sets is \textit{largest} if it is the largest element in $\mathcal{A}\mathcal{H}(G)$.
\end{definition}

\begin{definition}\cite[Section 3]{hyperbolicstructure}
We say two actions of a group $G$ on metric spaces $S$ and $R$ are \textit{equivalent}, denoted by $G \curvearrowright S \sim G \curvearrowright R$, if there exists a coarsely $G$-equivariant quasi-isometry $S \to R$.
\end{definition}

The action of a group $G$ on a metric space $S$ is said to be \textit{cobounded} if there exists a bounded subset $B \subseteq S$ such that $S = \bigcup_{g \in G} gB$. Given a cobounded acylindrical action of $G$ on a hyperbolic space $S$, the following lemma provides a (possibly infinite) generating set $X$ of $G$ such that $\Gamma(G,X)$ is equivariantly quasi-isometric to $S$.

\begin{lemma}\cite[Lemma 3.11]{hyperbolicstructure}\label{lemma: equivaction}
Let $G$ be a group acting coboundedly on a geodesic metric space $S$. Let $B \subseteq S$ be a bounded subset such that $\bigcup_{g \in G} gB = S$. Let $D =$ diam$(B)$ and let $b$ be any point of $B$. Then the group $G$ is generated by the set
\[X = \{g \in G\hspace{0.1cm}|\hspace{0.1cm} d_S(b, gb) \leq 2D + 1\}\]
and the natural action $G\curvearrowright \Gamma(G, X)$ is equivalent to $G \curvearrowright S$.
\end{lemma}

As an application of the above lemma, we say that a particular cobounded acylindrical action $G\curvearrowright S$ on a hyperbolic space is largest, when the equivalence class of the generating set associated to this action through the above correspondence, $[X]$, that is the largest element in $\mathcal{AH}(G)$.

We now focus on cobounded acylindrical tree actions. Let $(\Gamma,\mathcal{G})$ be a finite graph of groups, and let $T$ be the Bass-Serre tree of $(\Gamma,\mathcal{G})$. Denote the fundamental group of graph of groups $\pi_1(\Gamma,\mathcal{G})$ by $G$. We first note that $\Gamma$ can be lifted to a subtree $\tilde{\Gamma}$ in $T$ such that $G\cdot \tilde{\Gamma}=T$, i.e., $\tilde{\Gamma}$ is the fundamental domain of $T$ under the action $G\curvearrowright T$. Since $\Gamma$ is a finite graph, we have $\tilde{\Gamma}$ is of finite diameter and $G\curvearrowright T$ is a cobounded action. Let $D=\mbox{diam}(\tilde{\Gamma})$ and let $\tilde{v}_0$ be a base vertex of $\tilde{\Gamma}$. Then Lemma \ref{lemma: equivaction} provides a generating set $X=\{g\in G| d_T(\tilde{v}_0,g\tilde{v}_0)\leq 2D+1\}$ of $G$. Further, we have $G\curvearrowright T\sim G\curvearrowright \Gamma(G,X)$. Note that $G$, as the fundamental group of $(\Gamma,\mathcal{G})$, has the usual generating set $Y=(\sqcup_{v\in V(\Gamma)} G_v)\sqcup \{e\in E(\Gamma)| \ e \ \mbox{is a stable letter}\}$. The following lemma shows that these two generating sets $X$ and $Y$ are equivalent.

\begin{lemma}\label{lemma: eqGenSets}
Let $X$ and $Y$ be the generating sets of $G$ defined as above, then  $X \sim Y$.
\end{lemma}
\begin{proof}
From Bass-Serre theory, we note that $Y\subset X$. This implies that $\sup_{Y \in Y}|y|_X = 1$. Therefore $X \preceq Y$. Conversely, without loss of generality, we may assume that $\tilde{v}_0$ is labeled by the group $G_{v_0}$. As discussed in Remark 1.18 of \cite{Bass}, each vertex of $T$ at distance $n$ from $G_{v_0}$ can be represented by a unique reduced path in $\Gamma$ of the form $g_1e_1g_2e_2...g_ne_n$ where $o(e_1)=v_0$ and $g_i$ is contained in the transversal set $\Sigma_{e_i}$ for $G_{o(e_i)}/\alpha_{\overline{e}_i}(G_{e_i})$. Thus, for each $g\in X$ we have $g\tilde{v}_0=gG_{v_0}=g_1e_1g_2e_2...g_ne_nG_{v_0}$ where $n=d_T(\tilde{v}_0,g\tilde{v}_0)\leq 2D+1$. This implies that $g=g_1e_1g_2e_2...g_ne_nh$ for some $h\in G_{v_0}$, which shows that $g$ has word length at most $4D+3$ with respect to the generating set $Y$. Therefore we have $Y \preceq X$.

\end{proof}

Before giving the proof of our main theorem, we record the following proposition which gives sufficient conditions for an action to be the largest. Our goal is to apply this result to acylindrical tree actions.

\begin{proposition}\cite[Proposition 5.2]{largestaction}\label{prop: largestelt}
Let $G$ be a group, $\{H_1,\cdots, H_n\}$ a finite collection of subgroups of $G$, and $F$ be a finite subset of $G$ such that $F \cup (\bigcup_{i=1}^{n} H_i)$ generates $G$. Assume that:
\begin{enumerate}
    \item $\Gamma(G, F \cup (\bigcup_{i=1}^{n} H_i))$ is hyperbolic and the action of $G$ on it is acylindrical.
    \item Each $H_i$ is elliptic in every acylindrical action of $G$ on a hyperbolic space.
\end{enumerate}
Then $[F \cup (\bigcup\limits_{i=1}^{n} H_i)]$ is the largest element in $\mathcal{A}\mathcal{H}(G)$.
\end{proposition}

We are now ready to prove our main theorem in this section.

\begin{theorem}\label{thm: largestaction}
Let $(\Gamma, \mathcal{G})$ be a finite graph of groups, and let $T$ be the Bass-Serre tree of $(\Gamma, \mathcal{G})$. Denote the fundamental group $\pi_1(\Gamma, \mathcal{G})$ by $G$. If the action $G\curvearrowright T$ is acylindrical, and if each vertex group acts elliptically whenever $G$ acts acylindrically on any hyperbolic space $S$, then $G$ admits the largest acylindrical action on $T$.
\end{theorem}

\begin{proof}
Let $F$ be the set of stable letters of $\Gamma$, and let $\{H_1,...,H_n\}$ be the collection of vertex groups in $(\Gamma,\mathcal{G})$. Since $\Gamma$ is a finite graph, we have $F$ is finite and $\{H_1,...,H_n\}$ is a finite collection. Further, by Bass-Serre theory, we have $F \cup (\bigcup_{i=1}^{n} H_i)$ generates $G$. By Lemma \ref{lemma: eqGenSets}, we have the action $G\curvearrowright T$ is equivalent to $G\curvearrowright \Gamma(G, F \cup (\bigcup_{i=1}^{n} H_i))$. By Lemma 3.9 of \cite{hyperbolicstructure}, this is equivalent to $T\sim \Gamma(G, F \cup (\bigcup_{i=1}^{n} H_i))$. Thus there exists a coarsely $G$-equivariant quasi-isometry $T\to \Gamma(G, F \cup (\bigcup_{i=1}^{n} H_i))$. Note that hyperbolicity is preserved under quasi-isometries, and it is a fact that acylindricity of an action is preserved under equivariant quasi-isometries. Therefore we have $\Gamma(G, F \cup (\bigcup_{i=1}^{n} H_i))$ is hyperbolic and the action $G\curvearrowright \Gamma(G, F \cup (\bigcup_{i=1}^{n} H_i))$ is acylindrical. Thus the theorem follows directly from Proposition \ref{prop: largestelt} above.
\end{proof}

\begin{corollary}\label{coro: largestaction}
Let $(\Gamma,\mathcal{G})$ be a finite graph of groups, and let $T$ be the Bass-Serre tree of $(\Gamma,\mathcal{G})$. Denote the fundamental group $\pi_1(\Gamma,\mathcal{G})$ by $G$. If the action $G\curvearrowright T$ is acylindrical, and if each vertex group consists of finite order elements only, then $G$ admits the largest acylindrical action on $T$.
\end{corollary}

\begin{proof}
Since each vertex group consists of finite order elements only, then it can not contain any hyperbolic elements whenever $G$ acts acylindrically on any hyperbolic space $S$. By Theorem \ref{acylhypgrpclass}, this is equivalent to the statement that each vertex group acts elliptically whenever $G$ acts acylindrically on any hyperbolic space $S$. Thus the corollary follows directly from Theorem \ref{thm: largestaction} above.
\end{proof}

We note that the outer automorphism group of some Baumslag-Solitar groups is a family of examples that satisfies Corollary \ref{coro: largestaction} and thus Theorem \ref{thm: largestaction}. We will descirbe these groups more explicitly in subsection \ref{subsection: compare actions of OutBS}.

\section{Outer automorphism group of \texorpdfstring{$\mbox{BS}(p,q)$}{Gpq}}\label{section: OutBS}

Recall the Baumslag-Solitar group $\mbox{BS}(p, q)$ is defined to be the group with the following presentation
\[\mbox{BS}(p,q)=\langle x,t\ |\ tx^pt^{-1}=x^q\rangle\]
where $p,q\in \Z\setminus\{0\}$. By interchanging $t\leftrightarrow t^{-1}$, one may always assume $1 \leq p\leq |q|$. Throughout this section, we only consider the non-solvable Baumslag-Solitar groups i.e. we assume that $1 < |p| \leq |q|$. We note that presentations for these outer automorphism groups are known (see Remark \ref{remark: presentation of OutBS} and subsection \ref{subsection: two graphs of Out(BS)} for details).
\begin{remark}\label{remark: presentation of OutBS}
We have the following presentations for Out$(\mbox{BS}(p,q))$ when $p$ does not divide $q$ properly (details can be found in Section 3 of \cite{AutomorphismBSgroups} or \cite{pqnocommonfactors} and \cite{pdoesnotdivideq}).
\begin{itemize}
    \item[-] $\mbox{Out(BS}(p,q)) = \mathbb{Z}_{2|p-q|} \rtimes \mathbb{Z}_2$ if $p$ does not divide $q$.
    \item[-] $\mbox{Out(BS}(p,q)) = \mathbb{Z} \rtimes (\mathbb{Z}_2 \times \mathbb{Z}_2)$ if $p=q$.
    \item[-] $\mbox{Out(BS}(p,q)) = \mathbb{Z}_{2p} \rtimes \mathbb{Z}_2$ if $p = -q$.
\end{itemize}
\end{remark}
We note that Out$(\mbox{BS}(p,q))$ is either finite or virtually $\mathbb{Z}$ in any of the above cases, so it suffices to investigate the case when $p$ properly divides $q$. In this case, Clay showed that Out$(\mbox{BS}(p,q))$ admits a ray of groups structure $(\Gamma_1,\mathcal{G}_1)$ pictured as in Figure \ref{fig: Out(BS) ray}. Further, he showed that $(\Gamma_1,\mathcal{G}_1)$ can be collapsed to an edge of groups $(\Gamma_2,\mathcal{G}_2)$ pictured as in Figure \ref{fig: Out(BS) segment}, and he used $(\Gamma_2,\mathcal{G}_2)$ to compute the presentation of Out$(\mbox{BS}(p,q))$ via Bass-Serre theory. We will describe this presentation in subsection \ref{subsection: two graphs of Out(BS)}. In this section, our first goal is to compare the two tree actions that arise from these two graph of groups structures in the sense of Section \ref{section: largest section}. Next, we apply Corollary \ref{coro: quotient of amalgamation} to Out$(\mbox{BS}(p,q))$ and we show that the quotient action of a non-elementary acylindrical tree action of Out$(\mbox{BS}(p,q))$ by some equivariant family of subgroups is a non-elementary acylindrical tree action.

\underline{Assumption}. For the remainder of this section, we will assume that $q=pn$ where $p, |n|>1$.

\subsection{Deformation space of trees} We begin by recalling the definition of deformation space of trees, which is the tool Clay used to compute the presentation of Out$(\mbox{BS}(p,q))$.

\begin{definition}\cite[Definition 1.1]{AutomorphismBSgroups}
Let $(\Gamma,\mathcal{G})$ be a graph of groups. An edge $e\in E(\Gamma)$ is \textit{collapsible} if $o(e)\neq t(e)$ and $G_{t(e)}=\alpha_e(G_e)$. If one collapses $\{e,\overline{e}\}$ to the vertex $o(e)$, the resulting graph of groups is said to be obtained from $(\Gamma,\mathcal{G})$ by a \textit{collapse move}. The reverse of this move is called an \textit{expansion move}. A finite sequence of these moves is called an \textit{elementary deformation}. 
\end{definition}

\begin{remark}
We note that if an edge $e$ (not a loop) has $G_{t(e)}=\alpha_e(G_e)$, then $e$ contributes to a free amalgamated product $G_{o(e)}\ast_{G_e} G_{e}$ in $\pi_1(\Gamma, \mathcal{G})$. Therefore the above two moves correspond to the graph of groups isomorphism $G_{o(e)}\ast_{G_e} G_{e}\cong G_{o(e)}$. Thus we can always collapse this edge in $\Gamma$ without changing the fundamental group. See Figure \ref{fig:collapse}.
\end{remark}

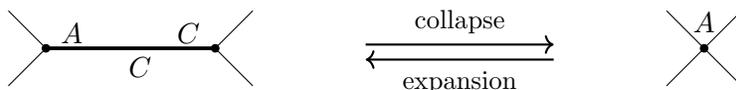
\begin{figure}[ht]
    \centering
\begin{tikzpicture}
      \tikzset{enclosed/.style={draw, circle, inner sep=0pt, minimum size=.1cm, fill=black}}
      
      \node[enclosed, label={right, yshift=.2cm: \small $A$}] at (2.25,3.25) {};
      \node[enclosed, label={left, yshift=.2cm: \small$C$}] at (4.5,3.25) {};
      \node[label={below, yshift=-0.3cm:}] at (3.5,3) {\small$C$}; 
      \node[enclosed] at (11,3.25) {};
       \node[label={below, yshift=-0.3cm:}] at (11,3.6) {\small$A$};

      \draw[line width=0.5mm] (2.25,3.25) -- (4.5,3.25) node[midway, sloped, above] {};
      \draw (2.25,3.25) -- (1.75, 3.75) node[midway, right] {};
      \draw (2.25,3.25) -- (1.75, 2.75) node[midway, right] {};
      \draw (4.5,3.25) -- (5, 3.75) node[midway, right] {};
      \draw (4.5,3.25) -- (5, 2.75) node[midway, right] {};

      \draw[->, thick] (6.5, 3.3) -- (9, 3.3) node[midway, above] {\footnotesize collapse};
      \draw[<-, thick] (6.5, 3.1) -- (9, 3.1) node[midway, below] {\footnotesize expansion};
      
      \draw (11,3.25) -- (10.5, 3.75) node[midway, left] {};
      \draw (11,3.25) -- (10.5, 2.75) node[midway, left] {}; 
      \draw (11,3.25) -- (11.5, 3.75) node[midway, right]{};
      \draw (11,3.25) -- (11.5, 2.75) node[midway, right]{};
\end{tikzpicture}
    \caption{Example of collapse and expansion moves.}
    \label{fig:collapse}
\end{figure}

\begin{definition}
Given $G$-tree $T$, the \textit{unnormalized deformation space} $\mathcal{X}$ containing $T$ is the set of all metric $G$-trees that have the same elliptic subgroups as $T$. We can projectivize $\mathcal{X}$ by taking the quotient of $\mathcal{X}$ under the action of $\mathbb{R}^{+}$ given by rescaling the metric on a given metric $G$-tree. The quotient $\mathcal{D}$ is called a \textit{deformation space} of $G$-trees.  
\end{definition}

Culler and Vogtmann's Outer space is a celebrated example of a deformation space for a finitely generated free group where the only elliptic subgroup is the trivial group. For now, the deformation space $\mathcal{D}$ is just a set, with no additional structure. However, we can view it as an infinite dimensional simplicial complex. In particular, each projectivized $G$-tree $T\in \mathcal{D}$ corresponds to a $n$-dimensional simplex where $n$ is the number of edges in the quotient graph $T/G$, and faces of this simplex correspond to projectivized $G$-trees $T'$ that can be obtained from $T$ by a sequence of collapse moves. We note that, by Forester's deformation theorem \cite[Theorem 1.1]{Foresterdeformation}, any two projectivized $G$-trees $T,T'\in \mathcal{D}$ (disregarding the metric) are related by an elementary deformation. Thus the deformation space $\mathcal{D}$ is in fact a connected simplicial complex. There is a natural action of Out$(G)$ on $\mathcal{D}$ given by precomposing of the action of $G$ on $T$ by an outer automorphism of $G$.

\subsection{Two graph of groups structures of \texorpdfstring{$\mbox{Out}(\mbox{BS}(p,q))$}{OutBSpq}}\label{subsection: two graphs of Out(BS)}
In this subsection, we describe the two graph of groups structures of Out$(\mbox{BS}(p,q))$, and we briefly explain how Clay computed the presentation of Out$(\mbox{BS}(p,q))$ (see Section 4.3 in \cite{AutomorphismBSgroups} for details).

Let $G_{p,q}$ denote the Baumslag-Solitar group $\mbox{BS}(p,q)$, and let $\mathcal{D}_{p,q}$ be the canonical deformation space of $G_{p,q}$.
First, Clay defined an $G_{p,q}$-invariant deformation retract $X_{p,q}$ of $\mathcal{D}_{p,q}$, where each vertex of $X_{p,q}$ corresponds to a particular type of reduced $G_{p,q}$-tree. He showed that \cite[Theorem 3.9]{AutomorphismBSgroups} this subcomplex $X_{p,q} \subseteq \mathcal{D}$ is in fact a tree. Furthermore, he showed that \cite[Proposition 4.1]{AutomorphismBSgroups} the quotient graph $X_{p,q}/G_{p,q}$ is an infinite ray as pictured in Figure \ref{fig: Out(BS) ray}. Thus, by Bass-Serre theory, the quotient graph $X_{p,q}/G_{p,q}$ admits a graph of groups structure $(\Gamma_1, \mathcal{G}_1)$ where $\Gamma_1=X_{p,q}/G_{p,q}$. We now discuss this graph of groups structure in details.
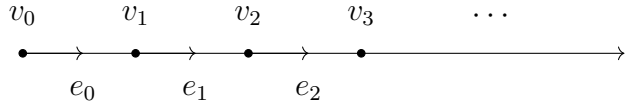
\begin{figure}[ht]
    \centering
\begin{tikzpicture}
      \node[label={below, yshift=-0.3cm:}] at (0,0.5) {$v_0$};
      \node[label={below, yshift=-0.3cm:}] at (1.5,0.5) {$v_1$};
      \node[label={below, yshift=-0.3cm:}] at (3,0.5) {$v_2$};
      \node[label={below, yshift=-0.3cm:}] at (4.5,0.5) {$v_3$};
      \node[label={below, yshift=-0.3cm:}] at (6.25,0.5) {$\cdots$};
      \node[label={below, yshift=-0.3cm:}] at (0.8,-0.5) {$e_0$};
      \node[label={below, yshift=-0.3cm:}] at (2.3,-0.5) {$e_1$};
      \node[label={below, yshift=-0.3cm:}] at (3.8,-0.5) {$e_2$};

      \tikzset{enclosed/.style={draw, circle, inner sep=0pt, minimum size=.1cm, fill=black}}
     
      \node[enclosed, label={right, yshift=.2cm:}] at (0,0) {};
      \node[enclosed, label={right, yshift=.2cm:}] at (1.5,0) {};
      \node[enclosed, label={right, yshift=.2cm:}] at (3,0) {};      \node[enclosed, label={right, yshift=.2cm:}] at (4.5,0) {};

      \draw[->] (0, 0) -- (8, 0);
      \draw[->] (0, 0) -- (0.8, 0);
      \draw[->] (1.5, 0) -- (2.3, 0);
      \draw[->] (3, 0) -- (3.8, 0);

\end{tikzpicture}
    \vspace*{-0.5cm}
    \caption{The graph of groups structure $(\Gamma_1,\mathcal{G}_1)$ of Out$(\mbox{BS}(p,q))$.}
    \label{fig: Out(BS) ray}
\end{figure}

Clay showed that the vertex group $G_{v_0}$ is isomorphic to the dihedral group $\mathbb{Z}_{p|n-1|} \rtimes \mathbb{Z}_2$, generated by the following automorphisms:
\begin{align*}
\psi: \hspace{0.5cm} & x \mapsto x \hspace{2cm} \iota: \hspace{0.5cm} x \mapsto x^{-1}\\
&t \mapsto xt \hspace{3cm}  t \mapsto t
\end{align*}
and for $k \geq 1$, the vertex group $G_{v_k}$ is isomorphic to the dihedral group $\mathbb{Z}_{|n^k(n-1)|} \rtimes \mathbb{Z}_2$ generated by the following automorphisms: 
\begin{align*}
\phi_{k}: \hspace{0.5cm} & x \mapsto x \hspace{2cm} \iota: \hspace{0.5cm} x \mapsto x^{-1}\\
&t \mapsto (t^{-k}x^pt^k)t \hspace{1.6cm}  t \mapsto t
\end{align*}
Notice that we have the following relations:
$\psi^{p(n-1)} = c_x^{-p}$ (where $c_x$ is the inner automorphism $g \mapsto xgx^{-1}$), $\iota \psi = \psi^{-1}\iota$, $\phi_{k+1}^n = \phi_k$ for $k \geq 1$, $\phi_1^n = \psi^p$ and $\iota\phi_k = \phi_k^{-1}\iota$. 

In summary, we have the vertex groups of $(\Gamma_1,\mathcal{G})$ are
$$G_{v_0}=\langle \psi, \iota\  | \ \psi^{p(n-1)}=\iota^2=1,\ \iota\psi=\psi^{-1}\iota \rangle$$
and 
$$G_{v_k}=\langle \phi_k, \iota\ |\ \iota^2=1,\ \phi_k^n=\phi_{k-1}, \ \iota\phi_k=\phi_k^{-1}\iota \rangle$$
The edge groups of $(\Gamma_1,\mathcal{G}_1)$ are 

$$\alpha_{e_0}(G_{e_0})=\langle\phi_1^n,\iota \rangle \leq G_{v_1}\ ,\ \alpha_{\overline{e}_0}(G_{e_0})=\langle\psi^p,\iota \rangle \leq G_{v_0}$$
and 
$$\alpha_{e_k}(G_{e_k})=\langle\phi_k, \iota\rangle\leq G_{v_{k+1}}\ ,\ \alpha_{\overline{e}_k}(G_{e_k})=G_{v_k}\ \ \mbox{for}\ k\geq 1$$

As graph of groups, the edge $e_0$ contributes to a free amalgamated product as follows.

$$G_{v_0}\ast_{G_{e_0}}G_{v_1}=(\mathbb{Z}_{|p(n-1)|} \rtimes \mathbb{Z}_2) \ast_{\mathbb{Z}_{|n-1|}\rtimes \mathbb{Z}_2} (\mathbb{Z}_{|n(n-1)|}\rtimes \mathbb{Z}_2)$$
and each $e_k$ for $k\geq 1$ contributes to a free amalgamated product as follows.
$$G_{v_k}\ast_{G_{e_k}}G_{v_{k+1}}=(\mathbb{Z}_{|n^k(n-1)|}\rtimes \mathbb{Z}_2)\ast_{\mathbb{Z}_{|n^k(n-1)|}\rtimes \mathbb{Z}_2} (\mathbb{Z}_{|n^{k+1}(n-1)|}\rtimes \mathbb{Z}_2)$$

We note that each $e_k$ for $k\geq 1$ is in fact a collapsible edge. So the infinite ray $\Gamma_1$ can collapse to a segment pictured as in Figure \ref{fig: Out(BS) segment}, where one vertex corresponding to $v_0$ and the other vertex corresponding to the end represented by $(v_1, v_2, \dotsc)$. 
\begin{figure}[ht]
    \centering
\begin{tikzpicture}
      \node[label={below, yshift=-0.3cm:}] at (0,0.5) {$v_0$};
      \node[label={below, yshift=-0.3cm:}] at (1.5,-0.5) {$e_0$};
      \node[label={below, yshift=-0.3cm:}] at (3.2,0.5) {$(v_1, v_2, \dotsc)$};

      \tikzset{enclosed/.style={draw, circle, inner sep=0pt, minimum size=.1cm, fill=black}}
     
      \node[enclosed, label={right, yshift=.2cm:}] at (0,0) {};
      \node[enclosed, label={right, yshift=.2cm:}] at (3,0) {};     

      \draw (0, 0) -- (3, 0);

\end{tikzpicture}
    \vspace*{-0.5cm}
    \caption{The graph of groups structure $(\Gamma_2,\mathcal{G}_2)$ of Out$(\mbox{BS}(p,q))$.}
    \label{fig: Out(BS) segment}
\end{figure}
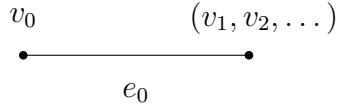

Since each $G_{v_k}$ is embedded in $G_{v_{k+1}}$ as a subgroup for $k\geq 1$, the stabilizer of this other end is the direct limit:
\[\lim_{\to} \mathbb{Z}_{n^k|n-1|} \rtimes \mathbb{Z}_2 = \mathbb{Z}[\frac{1}{|n|}]/|n(n-1)| \rtimes \mathbb{Z}_2\]
Thus the group Out$(G_{p,q})$ can be presented as a free amalgamated product as follows.
\begin{equation}\label{GrpPres}
\mbox{Out}(G_{p,q}) = (\mathbb{Z}_{|p(n-1)|}\rtimes \mathbb{Z}_2)\ast_{\mathbb{Z}_{|n-1|}\rtimes \mathbb{Z}_2}(\mathbb{Z}[\frac{1}{|n|}]/|n(n-1)\mathbb{Z}|\rtimes \mathbb{Z}_2)
\end{equation}

\subsection{Quotient action of \texorpdfstring{$\mbox{Out}(\mbox{BS}(p,q))$}{OutGpq}}\label{subsection: quotient action of OutBS} In this subsection, we apply Corollary \ref{coro: quotient of amalgamation} to Out$(G_{p,q})$. Denote the Bass-Serre tree associated to the graph of groups $(\Gamma_2,\mathcal{G}_2)$ by $T_{p,q}$. We begin by showing that the tree action $\mbox{Out}(G_{p,q})\curvearrowright T_{p,q}$ is non-elementary acylindrical. In fact, this follows from a more general result as follows.

\begin{proposition}\label{prop: amalgamationacyl}
Let $G=A\ast_D B$ be a free amalgamated product where $D$ is finite. If $[A:D]\geq 2$ and $[B:D]>2$, then the action of $G$ on its Bass-Serre tree $T$ is non-elementary acylindrical.
\end{proposition}

\begin{proof}
It follows from the definition that the action $G\curvearrowright T$ is $(k,C)$-acylindrical for $k=1$ and $C=|D|$. Thus the action is acylindrical follows directly from Theorem \ref{thm: acyl=kcacyl}. It remains to show that the action is non-elementary. Since $[A:D]\geq 2$ and $[B:D]>2$, there exist a non-trivial element $a\in A\setminus D$ and two different non-trivial elements $b_1,b_2\in B\setminus D$ such that $ab_1$ and $ab_2$ contribute to two hyperbolic elements. Since $b_1\neq b_2$, we have $b_1D,b_2D$ are two different edges in $T$. This shows that $\mbox{Axis}(ab_1)$ and $\mbox{Axis}(ab_2)$ represent two different branches of $T$ respectively. Thus $ab_1$ and $ab_2$ are two independent hyperbolic elements. Therefore the action $G\curvearrowright T$ is non-elementary.
\end{proof}

\begin{corollary}\label{acylActionOnTpq}
$\mbox{Out}(G_{p,q}) \curvearrowright T_{p,q}$ is a non-elementary acylindrical action. 
\end{corollary}

\begin{proof}
The proof follows directly from the presentation \ref{GrpPres} and Proposition \ref{prop: amalgamationacyl}.

\end{proof}

\begin{corollary}\label{Out(BS(p,q))acyl}
$\mbox{Out(BS}(p,q))$ is an acylindrically hyperbolic group if and only if $p$ divides $q$ properly.
\end{corollary}

\begin{proof}
The proof follows directly from Remark \ref{remark: presentation of OutBS} and Corollary \ref{acylActionOnTpq}.
\end{proof}

\begin{remark}
    Note that when $p$ divides $q$ properly, it follows from the group presentation of Out(BS($p,q$)) \ref{GrpPres} that the group is hyperbolic relative to the infinitely generated vertex group in the amalgamated product as the edge group is finite. Thus, by Remark \ref{remark: presentation of OutBS}, we have Out(BS($p,q$)) is relatively hyperbolic if BS($p,q$) is non-solvable.
\end{remark}

We now are ready to apply Corollary \ref{coro: quotient of amalgamation} to the action Out$(G_{p,q})\curvearrowright T_{p,q}$. Recall that the presentation \ref{GrpPres} of Out$(G_{p,q})$ is a free amalgamated product, we then denote $\mbox{Out}(G_{p,q})=A\ast_C B$ where $A,B$ and $C$ are subgroups in presentation \ref{GrpPres} accordingly. To have an equivariant family of subgroups, by Remark \ref{rmk: orbitrep of equi-family}, it suffices to choose proper normal subgroups $R_A$ in $A$ and $R_B$ in $B$. We provide two examples as follows. Consider the action $\mbox{Out}(G_{4,12})\curvearrowright T_{4,12}$. In Example \ref{example: infinite quotient}, we provide an equivariant family of subgroups $\{R_v\}$ of Out$(G_{4,12})$ that satisfies conditions of Corollary \ref{coro: quotient of amalgamation}, and thus shows that the quotient action Out$(G_{4,12})/\langle R_v\rangle\curvearrowright T_{2,4}/\langle R_v\rangle$ is non-elementary acylindrical and the quotient group is $\mbox{Out}(G_{4,12})/\langle R_v\rangle$ is acylindrically hyperbolic. In contrast, we provide an equivariant family of subgroups of Out$(G_{4,12})$ in Example \ref{example: finite quotient} that violates conditions of Corollary \ref{coro: quotient of amalgamation}. The quotient action is still acylindrical, however, it is an elementary acylindrical action.

\begin{figure}[ht]
    \centering
\begin{tikzpicture}
      \node[label={above, yshift=0cm:}] at (0,3.3) {$\cdot$};
      \node[label={above, yshift=0cm:}] at (-0.2,3.35) {$\cdot$};
      \node[label={above, yshift=0cm:}] at (-0.3,3.5) {$\cdot$};

      \node[label={above, yshift=0cm:}] at (-0.7,1.8) {$\cdot$};
      \node[label={above, yshift=0cm:}] at (-0.8,1.65) {$\cdot$};
      \node[label={above, yshift=0cm:}] at (-1,1.6) {$\cdot$};

      \node[label={above, yshift=0cm:}] at (0.3,0.5) {$\cdot$};
      \node[label={above, yshift=0cm:}] at (0.4,0.3) {$\cdot$};
      \node[label={above, yshift=0cm:}] at (0.35,0.1) {$\cdot$};
      
      \node[label={above, yshift=0cm:}] at (4.1,1.6) {$\cdot$};
      \node[label={above, yshift=0cm:}] at (3.9,1.65) {$\cdot$};
      \node[label={above, yshift=0cm:}] at (3.8,1.8) {$\cdot$};

      \tikzset{enclosed/.style={draw, circle, inner sep=0pt, minimum size=.1cm, fill=black}}
      
      \node[enclosed,black, label={right, yshift=.2cm:}] at (1,2) {};

      \node[enclosed,black, label={right, yshift=.2cm:}] at (0,3.7) {};
      \node[enclosed,black, label={right, yshift=.2cm:}] at (-1,2) {};
      \node[enclosed,black, label={right, yshift=.2cm:}] at (0,0.3) {};

      \node[enclosed,black, label={right, yshift=.2cm:}] at (0.1,4.5) {};
      \node[enclosed,black, label={right, yshift=.2cm:}] at (-0.3,4.5) {};
      \node[enclosed,black, label={right, yshift=.2cm:}] at (-0.6,4.35) {};     \node[enclosed,black, label={right, yshift=.2cm:}] at (-0.8,4) {};
      \node[enclosed,black, label={right, yshift=.2cm:}] at (-1.7,2.6) {};
      \node[enclosed,black, label={right, yshift=.2cm:}] at (-1.7,1.4) {};
      \node[enclosed,black, label={right, yshift=.2cm:}] at (0.1,-0.55) {};
      \node[enclosed,black, label={right, yshift=.2cm:}] at (-0.3,-0.55) {};
      \node[enclosed,black, label={right, yshift=.2cm:}] at (-0.6,-0.4) {};
      \node[enclosed,black, label={right, yshift=.2cm:}] at (-0.8,-0.1) {};
      \node[enclosed,black, label={right, yshift=.2cm:}] at (-2,2.25) {};
      \node[enclosed,black, label={right, yshift=.2cm:}] at (-2,1.75) {};

      \draw[very thick](-1,2) -- (-2,2.25) node[midway, right] {};
      \draw[very thick](-1,2) -- (-2,1.75) node[midway, right] {};
    
      \draw[very thick,red](1,2) -- (4,2) node[midway, right] {};
      \draw[very thick,red](1,2) -- (0,3.7) node[midway, right] {};
      \draw[very thick](1,2) -- (-1,2) node[midway, right] {};
      \draw[very thick](1,2) -- (0,0.3) node[midway, right] {};

      \draw[very thick](0,3.7) -- (-0.3,4.5) node[midway, right] {};
      \draw[very thick,red](0,3.7) -- (-0.6,4.35) node[midway, right] {};

      \draw[very thick,red](0,3.7) -- (0.1,4.5) node[midway, right] {};
      \draw[very thick](0,3.7) -- (-0.8,4) node[midway, right] {};
      \draw[very thick](-1,2) -- (-1.7,1.4) node[midway, right] {};
      \draw[very thick](-1,2) -- (-1.7,2.6) node[midway, right] {};
      \draw[very thick](0,0.3) -- (0.1,-0.55) node[midway, right] {};
      \draw[very thick](0,0.3) -- (-0.8,-0.1) node[midway, right] {};
      \draw[very thick](0,0.3) -- (-0.6,-0.4) node[midway, right] {};
      \draw[very thick](0,0.3) -- (-0.3,-0.55) node[midway, right] {};

      \node[enclosed,black, label={right, yshift=.2cm:}] at (4,2) {};
      \node[enclosed,black, label={right, yshift=.2cm:}] at (5,3) {};      \node[enclosed,black, label={right, yshift=.2cm:}] at (5.3,2.4) {};      \node[enclosed,black, label={right, yshift=.2cm:}] at (5.3,1.6) {};

      \node[enclosed,black, label={right, yshift=.2cm:}] at (5,1) {};

      \draw[very thick](4,2) -- (5,3) node[midway, right] {};
      \draw[very thick, red](4,2) -- (5,1) node[midway, right] {};

      \draw[very thick, red](4,2) -- (5.3,2.4) node[midway, right] {};
      \draw[very thick](4,2) -- (5.3,1.6) node[midway, right] {};

      \node[label={above, yshift=0cm:}] at (1.2,2.3) {\small$A$};
      \node[label={above, yshift=0cm:}] at (3.8,2.3) {\small$B$};
      \node[label={above, yshift=0cm:}] at (2.5,1.7) {\small$C$};
      \node[label={above, yshift=0cm:}] at (0.5,3.6) {\small$\psi B$};
      \node[label={above, yshift=0cm:}] at (-0.7,2.3) {\small$\psi^2 B$};
      \node[label={above, yshift=0cm:}] at (-0.2,0.65) {\small$\psi^3 B$};
      \node[label={above, yshift=0cm:}] at (5.5,3) {\small$\phi_i A$};
      \node[label={above, yshift=0cm:}] at (6.5,2.5) {\small$\phi_i^{|n^i(n-1)|-1} A$};
      \node[label={above, yshift=0cm:}] at (5.8,1.5) {\small$\phi_j A$};
      \node[label={above, yshift=0cm:}] at (6.2,0.9) {\small$\phi_j^{|n^j(n-1)|-1} A$};

      \node[label={above, yshift=0cm:}] at (0.3,4.8) {\small$\psi\phi_i A$};     \node[label={above, yshift=0cm:}] at (-1.1,4.6) {\small$\psi\phi_j A$};

      \node[label={above, yshift=0cm:}] at (-2.35,2.75) {\small$\psi^2\phi_i A$};  
      \node[label={above, yshift=0cm:}] at (-2.7,1.8) {\small$\psi^2\phi_j A$};

      \node[label={above, yshift=0cm:}] at (-1.4,-0.05) {\small$\psi^3\phi_i A$}; \node[label={above, yshift=0cm:}] at (-0.2,-0.85) {\small$\psi^3\phi_j A$};

      \node[label={above, yshift=0cm:}] at (-4.5,2) {$\cdots$};
      \node[label={above, yshift=0cm:}] at (8.5,2) {$\cdots$};

\end{tikzpicture}
    \caption{A local picture of the tree $T_{4,12}$. Every translate of $A$ has valence 4 and every translate of $B$ has infinite valence. The two red lines represent Axis($\psi\phi_i$) and Axis($\psi\phi_j$) respectively.}
    \label{fig: T24}
\end{figure}
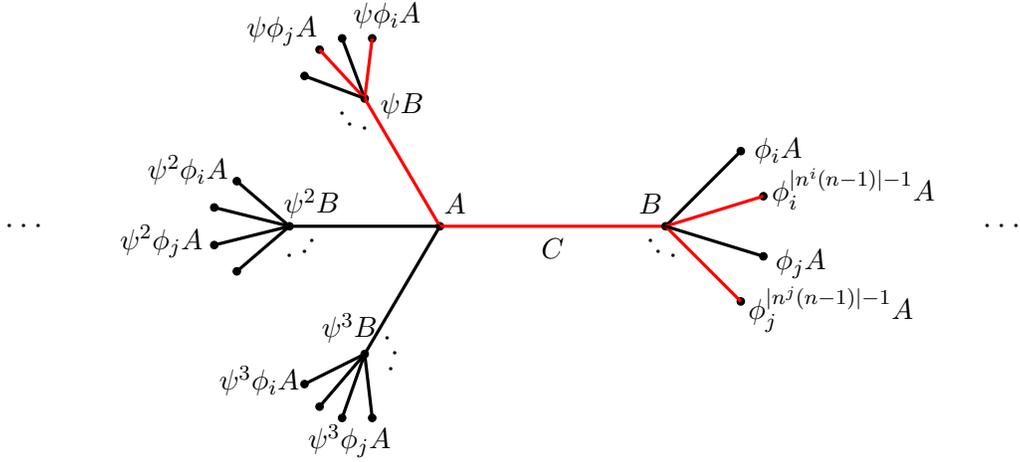

\begin{example}[Infinite quotient]\label{example: infinite quotient}
First, we have the following presentation of Out$(G_{4,12})$.
$$\mbox{Out}(G_{4,12})=A\ast_C B=(\mathbb{Z}_{8}\rtimes \mathbb{Z}_2)\ast_{\mathbb{Z}_{2}\rtimes \mathbb{Z}_2} (\mathbb{Z}[\frac{1}{3}]/6\mathbb{Z}\rtimes \mathbb{Z}_2)$$
where $A$ is generated by $\{\psi,\iota\}$, $B$ is generated by $\{\iota, \phi_1,...,\phi_k,...\}$. We also have the following relation $\psi^4=\phi_1^3$. It follows that $C$ is generated by $\{\psi^4,\iota\}$ (resp. $\{\phi_1^3,\iota\}$) as subgroup of $A$ (resp. $B$). Further, we have $\psi^8=\iota^2=1$ and $\phi_k^{3^k}=\phi_1$ for $k\geq 1$. This shows that all translates of $A$ in $T_{4,12}$ has valence $4$ whereas all translates of $B$ in $T_{4,12}$ has infinite valence. See Figure \ref{fig: T24} for the picture of $T_{4,12}$. Let us now consider $R_A$ is the subgroup of $A$ generated by $\{\psi^2,\iota\}$ and $R_B$ is the subgroup of $B$ generated by $\{\phi_1,\iota\}$. Since both $A$ and $B$ are abelian groups, we have $R_A$ and $R_B$ are both normal subgroups. One can verify that, under these choices, the equivariant family of subgroups $\{R_v\}$ satisfies all conditions of Corollary \ref{coro: quotient of amalgamation}. In fact, the quotient tree $T_{4,12}/\langle R_v \rangle$ is pictured as in Figure \ref{fig: infinite quotient}. In addition, each pair of elements $\psi\phi_i, \psi\phi_j$ for $i,j\geq 2$ and $i\neq j$ contribute to a pair of independent hyperbolic elements for the quotient action. Thus we have $\mbox{Out}(G_{4,12})/\langle R_v \rangle\curvearrowright T_{4,12}/\langle R_v \rangle$ is non-elementary acylindrical, and thus the quotient group $\mbox{Out}(G_{4,12})/\langle R_v \rangle$ is acylindrically hyperbolic.

\end{example}

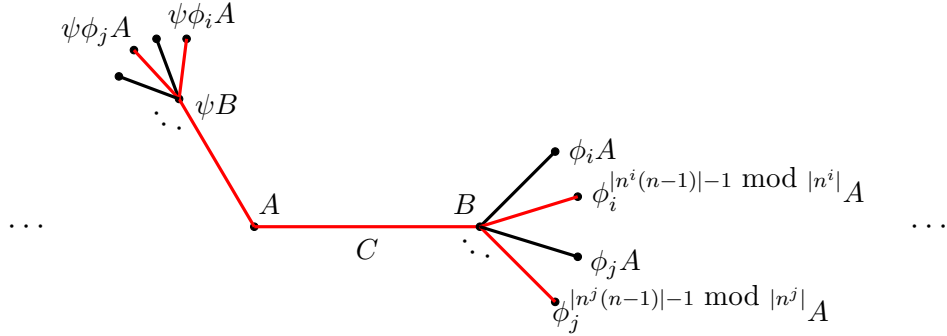
\begin{figure}[ht]
    \centering
\begin{tikzpicture}
          \node[label={above, yshift=0cm:}] at (0,3.3) {$\cdot$};
      \node[label={above, yshift=0cm:}] at (-0.2,3.35) {$\cdot$};
      \node[label={above, yshift=0cm:}] at (-0.3,3.5) {$\cdot$};

      \node[label={above, yshift=0cm:}] at (4.1,1.6) {$\cdot$};
      \node[label={above, yshift=0cm:}] at (3.9,1.65) {$\cdot$};
      \node[label={above, yshift=0cm:}] at (3.8,1.8) {$\cdot$};

      \tikzset{enclosed/.style={draw, circle, inner sep=0pt, minimum size=.1cm, fill=black}}
      
      \node[enclosed,black, label={right, yshift=.2cm:}] at (1,2) {};

      \node[enclosed,black, label={right, yshift=.2cm:}] at (0,3.7) {};
      \node[enclosed,black, label={right, yshift=.2cm:}] at (0.1,4.5) {};
      \node[enclosed,black, label={right, yshift=.2cm:}] at (-0.3,4.5) {};
      \node[enclosed,black, label={right, yshift=.2cm:}] at (-0.6,4.35) {};     \node[enclosed,black, label={right, yshift=.2cm:}] at (-0.8,4) {};

      \draw[very thick,red](1,2) -- (4,2) node[midway, right] {};
      \draw[very thick,red](1,2) -- (0,3.7) node[midway, right] {};

      \draw[very thick](0,3.7) -- (-0.3,4.5) node[midway, right] {};
      \draw[very thick,red](0,3.7) -- (-0.6,4.35) node[midway, right] {};

      \draw[very thick,red](0,3.7) -- (0.1,4.5) node[midway, right] {};
      \draw[very thick](0,3.7) -- (-0.8,4) node[midway, right] {};
  
      \node[enclosed,black, label={right, yshift=.2cm:}] at (4,2) {};
      \node[enclosed,black, label={right, yshift=.2cm:}] at (5,3) {};      \node[enclosed,black, label={right, yshift=.2cm:}] at (5.3,2.4) {};      \node[enclosed,black, label={right, yshift=.2cm:}] at (5.3,1.6) {};

      \node[enclosed,black, label={right, yshift=.2cm:}] at (5,1) {};

      \draw[very thick](4,2) -- (5,3) node[midway, right] {};
      \draw[very thick, red](4,2) -- (5,1) node[midway, right] {};

      \draw[very thick, red](4,2) -- (5.3,2.4) node[midway, right] {};
      \draw[very thick](4,2) -- (5.3,1.6) node[midway, right] {};

      \node[label={above, yshift=0cm:}] at (1.2,2.3) {\small$A$};
      \node[label={above, yshift=0cm:}] at (3.8,2.3) {\small$B$};
      \node[label={above, yshift=0cm:}] at (2.5,1.7) {\small$C$};
      \node[label={above, yshift=0cm:}] at (0.5,3.6) {\small$\psi B$};

      \node[label={above, yshift=0cm:}] at (5.5,3) {\small$\phi_i A$};
      \node[label={above, yshift=0cm:}] at (7.3,2.5) {\small$\phi_i^{|n^i(n-1)|-1\ \mbox{mod}\ |n^i|} A$};
      \node[label={above, yshift=0cm:}] at (5.8,1.5) {\small$\phi_j A$};
      \node[label={above, yshift=0cm:}] at (6.8,0.9) {\small$\phi_j^{|n^j(n-1)|-1\  \mbox{mod}\ |n^j|} A$};

      \node[label={above, yshift=0cm:}] at (0.3,4.8) {\small$\psi\phi_i A$};     \node[label={above, yshift=0cm:}] at (-1.1,4.6) {\small$\psi\phi_j A$};

      \node[label={above, yshift=0cm:}] at (-2,2) {$\cdots$};
      \node[label={above, yshift=0cm:}] at (10,2) {$\cdots$};

\end{tikzpicture}
    \caption{The infinite quotient tree $T_{4,12}/\langle R_v \rangle$ in Example \ref{example: infinite quotient}. The two red lines represent Axis($\psi\phi_i$) and Axis($\psi\phi_j$) respectively.}
    \label{fig: infinite quotient}
\end{figure}

\begin{example}[Finite quotient]\label{example: finite quotient}
Let us now consider $R_A$ is the subgroup of $A$ generated by $\{\psi^3,\iota\}$ and $R_B$ is the subgroup of $B$ still generated by $\{\phi_1,\iota\}$. Then we have $\langle C, R_A\rangle=A$ violates condition (1) of Corollary \ref{coro: quotient of amalgamation}. In this case, $T_{4,12}/\langle R_v \rangle$ is a finite diameter tree pictured as in Figure \ref{fig: infinite quotient}. The quotient action $\mbox{Out}(G_{4,12})/\langle R_v \rangle \curvearrowright T_{4,12}/\langle R_v \rangle$ is still acylindrical, however, there are no hyperbolic elements. Thus the quotient action is an elementary action.

\end{example}

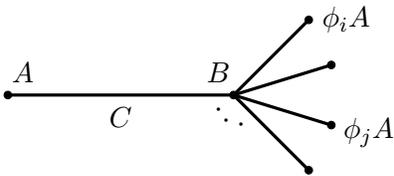
\begin{figure}[ht]
    \centering
\begin{tikzpicture}
      
      \node[label={above, yshift=0cm:}] at (4.1,1.6) {$\cdot$};
      \node[label={above, yshift=0cm:}] at (3.9,1.65) {$\cdot$};
      \node[label={above, yshift=0cm:}] at (3.8,1.8) {$\cdot$};

      \tikzset{enclosed/.style={draw, circle, inner sep=0pt, minimum size=.1cm, fill=black}}
      
      \node[enclosed,black, label={right, yshift=.2cm:}] at (1,2) {};

      \draw[very thick](1,2) -- (4,2) node[midway, right] {};
   
      \node[enclosed,black, label={right, yshift=.2cm:}] at (4,2) {};
      \node[enclosed,black, label={right, yshift=.2cm:}] at (5,3) {};      \node[enclosed,black, label={right, yshift=.2cm:}] at (5.3,2.4) {};      \node[enclosed,black, label={right, yshift=.2cm:}] at (5.3,1.6) {};

      \node[enclosed,black, label={right, yshift=.2cm:}] at (5,1) {};

      \draw[very thick](4,2) -- (5,3) node[midway, right] {};
      \draw[very thick](4,2) -- (5,1) node[midway, right] {};

      \draw[very thick](4,2) -- (5.3,2.4) node[midway, right] {};
      \draw[very thick](4,2) -- (5.3,1.6) node[midway, right] {};

      \node[label={above, yshift=0cm:}] at (1.2,2.3) {\small$A$};
      \node[label={above, yshift=0cm:}] at (3.8,2.3) {\small$B$};
      \node[label={above, yshift=0cm:}] at (2.5,1.7) {\small$C$};

      \node[label={above, yshift=0cm:}] at (5.5,3) {\small$\phi_i A$};
      \node[label={above, yshift=0cm:}] at (5.8,1.5) {\small$\phi_j A$};

\end{tikzpicture}
    \caption{The finite quotient tree $T_{4,12}/\langle R_v \rangle$ in Example \ref{example: finite quotient}.}
    \label{fig: finite quotient}
\end{figure}

\subsection{Comparing tree actions of \texorpdfstring{$\mbox{Out}(\mbox{BS}(p,q))$}{OutGpq}}\label{subsection: compare actions of OutBS}

In this subsection, we compare the two tree actions $\mbox{Out}(G_{p,q})\curvearrowright T_{p,q}$ and $\mbox{Out}(G_{p,q})\curvearrowright X_{p,q}$ in the sense of Section \ref{section: largest section}, and we provide a proof of our last main result, Theorem \ref{thmE}.

We begin by applying Theorem \ref{thm: largestaction} to the action $\mbox{Out}(G_{p,q})\curvearrowright T_{p,q}$ and thus showing that this action is the largest acylindrical action.

\begin{proposition}\label{prop: Tpq largest}
Out$(G_{p,q})$ admits the largest acylindrical action on the Bass-Serre tree $T_{p,q}$.
\end{proposition}

\begin{proof}
Note that Proposition \ref{acylActionOnTpq} shows that the action of Out$(G_{p,q})$ on $T_{p,q}$ is non-elementary acylindrical. Since the vertex groups contain finite order elements only, then the hypotheses of Corollary \ref{coro: largestaction} are satisfied and hence the proof follows.
\end{proof}

Now we turn our attention to the other tree action $\mbox{Out}(G_{p,q})\curvearrowright X_{p,q}$. We first give a short description of the tree $X_{p,q}$ (full details can be found in Section 3.2 of \cite{AutomorphismBSgroups}). By Bass-Serre theory, we can lift $\Gamma_1$ to an infinite ray $\tilde{\Gamma}_1$ in $X_{p,q}$ such that $\mbox{Out}(G_{p,q})\cdot \tilde{\Gamma}_1=X_{p,q}$. We denote the lift of $v_k$ in $\tilde{\Gamma}_1$ by $\tilde{v}_k$. Thus each vertex in $X_{p,q}$ is actually a translation of some $\tilde{v}_k$. We say a vertex has $\textit{level k}$ if it is a translation of $\tilde{v}_k$. By definition, the action $\mbox{Out}(G_{p,q})\curvearrowright X_{p,q}$ preserves the level of each vertex. Recall that the edge group $G_{e_k}$ has index $|n|$ in $G_{v_{k+1}}$ and index $1$ in $G_{v_{k}}$, so each level $k$ vertex in $X_{p,q}$ has valence $|n|+1$. Since $G_{e_0}$ has index $p$ in $G_{v_0}$, then each level $0$ vertex in $X_{p,q}$ has valence $p$. For example, to highlight the level of each vertex, we have a piece of the tree $X_{2,4}$ pictured as in Figure \ref{fig: X24}. The entire tree $X_{2,4}$ can be obtained by gluing 2 pieces along a level $0$ vertex, and passing to the new tree constructed by the gluing process.

\begin{figure}[ht]
    \centering
\begin{tikzpicture}
      \node[label={right, yshift=0cm:}] at (0,-0.4) {$\Tilde{v}_0$};
      \node[label={right, yshift=0cm:}] at (1.5,-0.4) {$\phi_1\Tilde{v}_0$};
      \node[label={right, yshift=0cm:}] at (3,-0.4) {$\phi_2\Tilde{v}_0$};
      \node[label={right, yshift=0cm:}] at (4.5,-0.4) {$\phi_2\phi_1\Tilde{v}_0$};
      \node[label={right, yshift=0cm:}] at (1.7,2) {$\Tilde{v}_1$};
      \node[label={right, yshift=0cm:}] at (4.2,2) {$\phi_2\Tilde{v}_1$};
      \node[label={above, yshift=0cm:}] at (3,4) {$\Tilde{v}_2$};
      \node[label={above, yshift=0cm:}] at (2.8,4.3) {$\cdot$};
      \node[label={above, yshift=0cm:}] at (2.9,4.45) {$\cdot$};
      \node[label={above, yshift=0cm:}] at (3,4.6) {$\cdot$};
      \node[label={below, yshift=0cm:}] at (-1.5,0) {Level 0};
      \node[label={below, yshift=0cm:}] at (-1.5,2) {Level 1};
      \node[label={below, yshift=0cm:}] at (-1.5,4) {Level 2};

      \tikzset{enclosed/.style={draw, circle, inner sep=0pt, minimum size=.1cm, fill=black}}
      
      \node[enclosed, label={right, yshift=.2cm:}] at (0,0) {};
      \node[enclosed, label={right, yshift=.2cm:}] at (1.5,0) {};
      \node[enclosed, label={right, yshift=.2cm:}] at (3,0) {};
      \node[enclosed, label={right, yshift=.2cm:}] at (4.5,0) {};

      \draw(0,0) -- (2.62, 4) node[midway, right] {};

      \draw(1.5,0) -- (1.3,2) node[midway, right] {};
      \draw(3,0) -- (3.55,2) node[midway, right] {};
      \draw(4.5,0) -- (2.62,4) node[midway, right] {};

\end{tikzpicture}
    \caption{A piece of the Bass-Serre tree $X_{2,4}$.}
    \label{fig: X24}
\end{figure}
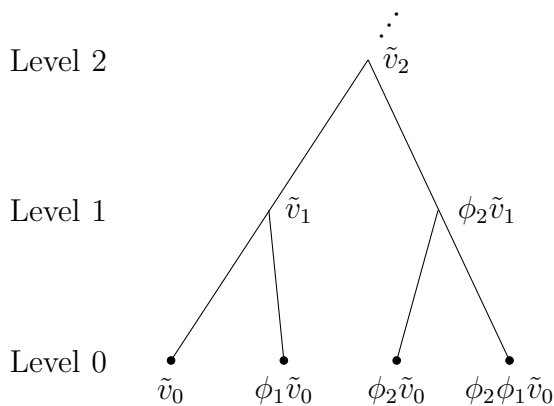

We are ready now to show that the action  $\mbox{Out}(G_{p,q})\curvearrowright X_{p,q}$ is not an acylindrical action.

\begin{proposition}\label{NotAcyl}
The action $\mbox{Out}(G_{p,q})\curvearrowright X_{p,q}$ is not acylindrical.
\end{proposition}

\begin{proof}
By definition \ref{acyldef}, we need to show that there exists an $\varepsilon > 0$ such that for any $R, N >0$, there exists $x,y$ in $X_{p,q}$ with $d(x,y)\geq R$ and
    \[|\{ g \in \mbox{Out}(G_{p,q}) \hspace{0.1cm}|\hspace{0.1cm} d(x, gx) \leq \varepsilon, \hspace{0.1cm} d(y, gy) \leq \varepsilon \}|  > N\]

Consider $\varepsilon = \frac{1}{2}$. Let $R, N>0$ be fixed. We can choose $k>0$ such that $|n^k(n-1)|>N$. Let $\tilde{v}_l$ be a vertex in $\tilde{\Gamma}_1$ such that $l>k$ and $d(\tilde{v}_l, \tilde{v}_k)\geq R$. Then, by choosing $x = \tilde{v}_l$ and $y = \tilde{v}_k$, we have
  
  \[|\{ g \in \mbox{Out}(G_{p,q})\hspace{0.1cm} |\hspace{0.1cm} d(\tilde{v}_l, g\tilde{v}_l) \leq \varepsilon, \hspace{0.1cm} d(\tilde{v}_k, g\tilde{v}_k) \leq \varepsilon \}| = |G_{v_k}|>|n^k(n-1)| > N\]
Hence the action is not acylindrical.
\end{proof}

We next turn toward a brief discussion of weakly properly discontinuously action (WPD action). The definition of this type of action is essentially due to Bestvina and Fujiwara \cite{BestvinaFujiwaraWPD}. However, to obtain clearer and simpler proofs, we will follow an equivalent definition in \cite{acylactionontrees}.

\begin{definition}\cite[Definition 3.5]{acylactionontrees}\label{WPDdef}
An element $h$ of a group $G$ acting isometrically on a metric space $S$ satisfies the \textit{weak proper discontinuity} condition (or $h$ is a \textit{WPD element}) if for some $s\in S$ (or, equivalently, for all $s\in S$) and every $\varepsilon\geq 0$, there is an $M \geq 0$ such that
\[|\{g \in G| d(s, gs) \leq \varepsilon, \hspace{0.1cm} \hspace{0.1cm}  d(h^Ms, gh^Ms) \leq \varepsilon \}| < \infty\]
\end{definition}

\begin{definition}
An action of a group $G$ on a hyperbolic space $S$ isometrically is called \textit{non-elementary WPD} if there are at least two independent hyperbolic elements that satisfy the WPD condition.
\end{definition}

The following theorem by Osin gives us different ways to prove a group to be acylindrically hyperbolic.

\begin{theorem} \cite[Theorem 1.2]{acylhypgps}\label{acylhypgpseq}
For any group $G$, the following are equivalent.
\begin{enumerate}
\item There exists a generating set $X$ of $G$ such that corresponding Cayley graph $\Gamma(G, X)$ is hyperbolic, $|\Gamma(G, X)|>2$, and the natural action of $G$ on $\Gamma(G, X)$ is acylindrical.
\item $G$ admits a non-elementary acylindrical action on a hyperbolic space.
\item $G$ is not virtually cyclic and admits an action on a hyperbolic space such that at least one element of $G$ is loxodromic and satisfies the WPD condition.
\item $G$ contains a proper infinite hyperbolically embedded subgroup.
\end{enumerate}
\end{theorem}

It is clear that a non-elementary acylindrical action is a non-elementary WPD action. However the converse is not true in general. We have shown in Proposition \ref{NotAcyl} that the action $\mbox{Out}(G_{p,q})\curvearrowright X_{p,q}$ is not acylindrical. Next, we show that this action is in fact a non-elementary WPD action.

\begin{lemma}\label{lemma: WPD condition}
For any level $0$-vertex $w$ in $X_{p,q}$, the set $\{g\in \mbox{Out}(G_{p,q})\hspace{0.1cm} | \hspace{0.1cm} g\tilde{v}_0=w\}$ is finite.
\end{lemma}
\begin{proof}
For simplicity, we denote $S_w=\{g\in \mbox{Out}(G_{p,q})\hspace{0.1cm} | \hspace{0.1cm} g\tilde{v}_0=w\}$. Let $m=d(w,\tilde{v}_0)$. As discussed in the proof of Theorem 1.17 of \cite{Bass}, $w$ can be uniquely represented by $g_1f_1g_2f_2...g_mf_mG_{v_0}$ where $(f_1,f_2,...,f_m)$ is a reduced edge path in $\Gamma_1$ with $o(f_1)=t(f_m)=v_0$, and $g_i$ is contained in the transversal set $\Sigma_{f_i}$ for $G_{o(f_i)}/\alpha_{\overline{f}_i}(G_{f_i})$. If $g\in S_w$, we have $gG_{v_0}=g_1f_1g_2f_2...g_mf_mG_{v_0}$ which implies that $g\in g_1f_1g_2f_2...g_mf_mG_{v_0}$.
Since every vertex group $G_{v_k}$ and every edge group $G_{e_k}$ in $\Gamma_1$ is finite, then there are only finitely many choices for $\Sigma_{f_i}$ and each $|\Sigma_{f_i}|$ is finite. Thus $g$ has only finitely many choices, which implies that $S_w$ is finite.
\end{proof}

\begin{proposition}\label{prop: WPD action}
The action $\mbox{Out}(G_{p,q})\curvearrowright X_{p,q}$ is a non-elementary WPD action.
\end{proposition}
\begin{proof}
For any $\varepsilon \geq 0$. Let $s=\tilde{v}_0\in X_{p,q}$. We first note that, similar to the proof of Proposition \ref{prop: amalgamationacyl}, every pair of elements $\psi\phi_i$ and $\psi\phi_j$ for $i,j\geq 1$ and $i\neq j$ contributes to a pair of independent hyperbolic elements for the action $\mbox{Out}(G_{p,q})\curvearrowright X_{p,q}$. We then claim that each $\psi\phi_i$ satisfies WPD condition. By Definition \ref{WPDdef}, it suffices to show that
\[|\{g \in \mbox{Out}(G_{p,q})\hspace{0.1cm}|\hspace{0.1cm} d(\tilde{v}_0, g\tilde{v}_0) \leq \varepsilon \}| < \infty\]
Note that there are only finitely many translates of $\tilde{v}_0$ inside a $\varepsilon$-neighborhood of $\tilde{v}_0$. Thus, by Lemma \ref{lemma: WPD condition}, the set $\{g \in \mbox{Out}(G_{p,q})\hspace{0.1cm}|\hspace{0.1cm} d(\tilde{v}_0, g\tilde{v}_0) \leq \varepsilon \}$ is finite. This completes the proof.
\end{proof}

We remark that, by Theorem \ref{acylhypgpseq}, the proposition above provides an alternative proof of acylindrical hyperbolicity of Out$(\mbox{BS}(p,q))$. We are ready now to prove our last main result, Theorem \ref{thmE}, which can be restated as follows.
\begin{theorem}
Let $G=Out(G_{p,q})$ where $q=pn$ for some $p,|n|>1$. Then we have the following results:
\begin{enumerate}
    \item $G \curvearrowright X_{p,q}$ is not an acylindrical action while $G \curvearrowright T_{p,q}$ is non-elementary acylindrical, thus $G \curvearrowright X_{p,q}$ is not equivalent to $G \curvearrowright T_{p,q}$ ;
    \item $G \curvearrowright X_{p,q}$ and $G \curvearrowright T_{p,q}$ are both non-elementary WPD actions;
    \item $G \curvearrowright X_{p,q}$ and $G \curvearrowright T_{p,q}$ have the same set of hyperbolic elements;
    \item $G \curvearrowright T_{p,q}$ is the largest acylindrical hyperbolic action.  
\end{enumerate}
\end{theorem}

\begin{proof}

(1), (2) and (4) follows directly from Corollary \ref{acylActionOnTpq}, Proposition \ref{NotAcyl}, Proposition \ref{prop: Tpq largest} and Proposition \ref{prop: WPD action}. For (3), we note that $T_{p,q}$ can be obtained from $X_{p,q}$ by a sequence of $G$-equivariant collapse moves. This implies that an element $g\in G$ is elliptic in $G\curvearrowright X_{p,q}$ if and only if $g$ is elliptic in $G\curvearrowright T_{p,q}$. Thus the two actions $G \curvearrowright X_{p,q}$ and $G \curvearrowright T_{p,q}$ have the same set of elliptic elements. It follows that they have the same set of hyperbolic elements.

\end{proof}

\bibliographystyle{alpha}
\bibliography{citations}

\end{document}